\documentclass{amsart}
\usepackage{mathrsfs,amssymb,mathrsfs, multirow,setspace}
\usepackage[a4paper, total={7in, 9in}]{geometry}
\usepackage{enumerate}
\theoremstyle{plain}
\usepackage{graphicx}
\newtheorem{theorem}{Theorem}[section]
\newtheorem{lemma}[theorem]{Lemma}
\newtheorem{proposition}[theorem]{Proposition}
\newtheorem{corollary}[theorem]{Corollary}

\theoremstyle{definition}

\theoremstyle{remark}
\newtheorem{remark}[theorem]{Remark}

\input xy
\xyoption{all}

\def\G{\Gamma}

\begin{document}
	\title[ On the inclusion ideal graph of semigroups]{On the inclusion ideal graph of semigroups}
	\author[Barkha Baloda, Jitender Kumar]{Barkha Baloda, $\text{Jitender Kumar}^{^*}$}
	\address{Department of Mathematics, Birla Institute of Technology and Science Pilani, Pilani, India}
	\email{barkha0026@gmail.com,jitenderarora09@gmail.com}

	\begin{abstract}
		The inclusion ideal graph $\mathcal{I}n(S)$ of a semigroup $S$ is an undirected simple graph whose vertices are all nontrivial left ideals of $S$ and two distinct left ideals $I, J$ are adjacent if and only if either $I \subset J$ or $J \subset I$. The purpose of this paper is to study algebraic properties of the semigroup $S$ as well as graph theoretic properties of $\mathcal{I}n(S)$. In this paper, we investigate the connectedness of $\mathcal{I}n(S)$. We show that diameter of $\mathcal{I}n(S)$ is at most $3$ if it is connected. 
		We also obtain a  necessary and sufficient condition of $S$ such that the clique number of $\mathcal{I}n(S)$ is $n$, where $n$ is the number of minimal left ideals of $S$. Further, various graph invariants of $\mathcal{I}n(S)$ viz. perfectness, planarity, girth etc. are discussed. For a completely simple semigroup $S$, we investigate various properties of $\mathcal{I}n(S)$ including its independence number and matching number. Finally, we obtain the automorphism group of $\mathcal{I}n(S)$. 
		\end{abstract}

	
	\keywords{Semigroup, ideals, completely simple semigroup,  graph automorphism\\ *  Corresponding author}
	
	\maketitle
	\section{Introduction and Motivation}
Literature is abound with numerous remarkable results concerning a number of constructions of  graphs from rings, semigroups or groups \cite{abdollahi2006non, i.Bosak, cameron2011power,  demeyer2002zerodivisor, transitivekelarev, combinatorialkelrav, a.Cayley-abundant, redmond2003ideal}. The investigation of graphs related to various algebraic structures is very important because graphs of this type have valuable applications and are related to automata theory(see \cite{b.kelarev2003graph, a.kelarev2009cayley, kelarevminimalautomata}). Akbari \emph{et al.} \cite{akbari2014inclusion} have introduced the notion of inclusion ideal graph associated with ring structure. 
The \emph{inclusion ideal graph} $\mathcal{I}n(R)$ of a ring $R$ is an undirected simple graph whose vertex set is the collection of nontrivial left ideals of  $R$ and two distinct nontrivial left ideals $I, J$ are adjacent if and only if either $I \subset J$ or $J \subset I$. Further, in \cite{ inclusionidealringakbari} Akbari \emph{et al.} have studied various graph invariants including  connectedness, perfectness, diameter, girth  of $\mathcal{I}n(R)$. It was shown that $\mathcal{I}n(R)$ is disconnected if and only if $R \cong M_{2}(D)$ or $D_1 \times D_2$, for some division rings, $D, D_1$ and $D_2$. The subspace inclusion graph $\mathcal{I}n(\mathbb{V})$ associated with vector space $\mathbb{V}$ has been studied by A. Das \cite{das2016}. The \emph{subspace inclusion graph} on a finite dimensional vector space $\mathbb{V}$ is undirected simple graph whose vertices are all nontrivial proper subspaces of $\mathbb{V}$ and two distinct nontrivial proper subspaces $W_1$ and $W_2$ are adjacent if and only if $W_1 \subset W_2$ or $W_2 \subset W_1$. In \cite{das2016}, A. Das has studied the diameter, girth, clique number and chromatic number of $\mathcal{I}n(\mathbb{V})$. Moreover, other graph invariants namely perfectness, planarity of $\mathcal{I}n(\mathbb{V})$ have been studied in \cite{subspaceinclusiondas}. Also, for a $3$-dimensional vector space it was shown that $\mathcal{I}n(\mathbb{V})$ is  bipartite, vertex transitive, edge transitive and has a perfect matching. D. Wong \emph{et al.} \cite{D.Wang2018} have proved the following conjectures proposed by A. Das \cite{subspaceinclusiondas}: If $\mathbb{V}$ is a 3-dimensional vector space over a finite field $F_q$ with $q$ elements, then
\begin{enumerate}
	\item[(1)] the domination number of $\mathcal{I}n(\mathbb{V})$ is $2q$.
	\item[(2)] $\mathcal{I}n(\mathbb{V})$ is distance regular.
\end{enumerate}
The problem to determine the independence number of vector space $\mathbb{V}$ when the base field is finite is solved by X. Ma \emph{et al.} \cite{independencenowang}. Further, the automorphism group of $\mathcal{I}n(\mathbb{V})$ was obtained by X. Wang \emph{et al.} in \cite{automorphismwang}.
Analogously, the subgroup inclusion graph of a group $G$, denoted by $\mathcal{I}n(G)$, has been studied by Devi and Rajkumar \cite{devi2016inclusion}. They classified the finite groups whose inclusion graph is complete, bipartite, tree, star, path, cycle, disconnected and claw-free. S. Ou \emph{et al.} \cite{ou2019diameters} determined the diameter of $\mathcal{I}n(G)$ when $G$ is nilpotent group and characterized the independent dominating sets as well as  automorphism group of $\mathcal{I}n(\mathbb{Z}_n)$. Further, S. Ou \emph{et al.} \cite{ou2019planarity} studied planarity and the fixing number of inclusion graph of a nilpotent group. 

It is pertinent as well as interesting to associate graphs to ideals of a semigroup as ideals gives a lot of information about structure of semigroups \cite{saito1958semigroups, satyanarayana1978structure}. Motivated with the work of \cite{akbari2014inclusion,das2016}, in this paper, we consider the inclusion ideal graph associated with semigroups. The \emph{inclusion ideal graph} $\mathcal{I}n(S)$ of a semigroup $S$ is an undirected simple graph whose vertex set is nontrivial left ideals of  $S$ and two distinct nontrivial left ideals $I, J$ are adjacent if and only if $I \subset J$ or $J \subset I$. The paper is arranged as follows. In Section 2, we state necessary fundamental notions  and recall some necessary results. Section 3 comprises the results concerning the inclusion ideal graph of an arbitrary semigroup. In Section 4, we study various graph invariants of  $\mathcal{I}n(S)$, where $S$ is a completely simple semigroup. Further, in Section 5, for a completely simple semigroup $S$,  the automorphism group of $\mathcal{I}n(S)$ is obtained. 
	\section{Preliminaries}
	In this section, first we recall necessary definitions and results of semigroup theory from \cite{b.clifford61vol1}. A \emph{semigroup} $S$ is a non-empty set together with an associative binary operation on $S$. The Green's $\mathcal{L}$-relation on a semigroup $S$ defined as $x$         $\mathcal{L}$ $y \Longleftrightarrow S^{1}x = S^{1}y$ where $S^{1}x = Sx \cup    \{x\}$. 
	 The $\mathcal{L}$-class of an element $a \in S$ is denoted as $L_a$. 	A non-empty subset $I$ of $S$ is said to be a \emph{left [right] ideal} if $SI \subseteq I  [IS \subseteq I]$ and an  \emph{ideal} of $S$ if $SIS \subseteq I$. Union of two left [right] ideals of $S$ is again a left [right] ideal of $S$. A left ideal $I$ is \emph{maximal} if it does not contained in  any nontrivial left ideal of $S$. If $S$ has unique maximal left ideal then it contains every nontrivial left ideal of $S$.
	A left ideal $I$ is \emph{minimal} if it does not properly contain any left ideal of $S$. It is well known that every non-zero element of a minimal left ideal of $S$ is in same $\mathcal{L}$-class. If $S$ has a minimal left ideal then every nontrivial left ideal contains at least one minimal left ideal. If $A$ is any other left ideal of $S$ other than minimal left ideal  $I$, then either $I \subset A$ or $I \cap A = \emptyset$. Thus we have the following remark.
	\begin{remark}\label{disjoint intersection minimal}
		Any two different minimal left ideals of a semigroup $S$ are disjoint.
	\end{remark} 
The following lemma is useful in the sequel and we shall use this without referring to it explicitly.
	\begin{lemma}\label{S minus K is lclass}
A left ideal $K$ of $S$ is maximal if and only if $S \setminus K$ is an $\mathcal{L}$-class.
\end{lemma}
\begin{proof}
First suppose that $S \setminus K$ is an $\mathcal{L}-$class. Let if possible $K$ is not maximal left ideal of $S$. Then there exists a nontrivial left ideal $K'$ of $S$ such that $K \subset K'$. There exists $a \in K'$ but $a \notin K$. Thus,  $L_a = S \setminus K$. Consequently, $L_a \subset K'$ gives $S = K'$, a contradiction. 
Conversely, suppose that $K$ is a maximal left ideal of $S$. For each $a \in S \setminus K$, maximality of $K$ implies $K \cup S^1a = S$. Consequently, $a$  $\mathcal{L}$  $b$ for every $a, b \in S \setminus K$. Thus $S \setminus K$ is contained in some $\mathcal{L}-$class and this $\mathcal{L}-$class is disjoint from $K$. It follows that $S \setminus K$ is an $\mathcal{L}-$class.   
\end{proof}


A semigroup $S$ is said to be \emph{simple} if it has no proper ideal. Let $\mathcal{E}$ be the set of idempotents of a semigroup $S$. If $e, f \in \mathcal{E}$, we define $e \leq f$ to mean $ef =fe= e.$ Recall that a semigroup $S$ is called \emph{completely simple} if $S$ is  simple  and contains a primitive idempotent. By \emph{primitive idempotent} we mean an idempotent which is minimal within the set of all idempotents under the relation $\leq$.
	\begin{lemma}[{ \cite[Corollary 2.49]{b.clifford61vol1}}]
	A completely simple semigroup is the union of its minimal left [right] ideals.
	\end{lemma}
	The following lemma can be proved easily.
	\begin{lemma}\label{union minimal}
	       Let $S$ be a completely simple semigroup with $n$ minimal left ideals. Then any nontrivial left ideal of $S$ is union of these minimal left ideals.
	\end{lemma}
	
	We also require the following graph theoretic  notions \cite{westgraph}. A \emph{graph} $\Gamma$ is a pair  $\Gamma = (V, E)$, where $V = V(\Gamma)$ and $E = E(\Gamma)$ are the set of vertices and edges of $\Gamma$, respectively. We say that two different vertices $u, v$ are $\mathit{adjacent}$, denoted by $u \sim v$ or $(u,v)$, if there is an edge between $u$ and $v$. We write $u \nsim v$, if there is no edge between $u$ and $v$. The \emph{distance} between two vertices $u, v$ in $\Gamma$ is the number of edges in a shortest path connecting them and it is denoted by $d(u, v)$. If there is no path between $u$ and $v$, we say that the distance between  $u$ and $v$ is \emph{infinity} and we write as $d(u, v) = \infty$. The diameter $diam(\Gamma)$ of $\Gamma$ is the greatest distance between any pair of vertices. The \emph{degree} of the vertex $v$ in $\Gamma$ is the number of edges incident to $v$ and it is denoted by $deg(v)$. 
	A graph $\Gamma$ is \emph{regular} if degree of every vertex is same. A graph $\Gamma$ is said to be \emph{biregular} if all vertices have two distinct degrees. A \emph{cycle} is a closed walk with distinct vertices except for the initial and end vertex, which are equal and a cycle of length $n$ is denoted by $C_n$. A graph $\G$ is \emph{bipartite}  if $V(\Gamma)$ is the union of two disjoint independent  sets.
	It is well known that a graph is bipartite if and only if it has no odd cycle {\cite[Theorem 1.2.18]{westgraph}}.
	The \emph{girth} of $\Gamma$ is the length of its shortest cycle and is denoted by ${g(\Gamma)}$. A connected graph $\Gamma$ is Eulerian if and only if degree of every vertex is even {\cite[Theorem 1.2.26]{westgraph}}. A \emph{subgraph}  of $\Gamma$ is a graph $\Gamma'$ such that $V(\Gamma') \subseteq V(\Gamma)$ and $E(\Gamma') \subseteq E(\Gamma)$.  A subgraph $\Gamma'$ of  $\Gamma$ is called an \emph{induced subgraph} by the elements of  $V(\Gamma') \subseteq V(\Gamma)$ if for $u, v \in V(\Gamma')$, we have $u \sim v$   in $\Gamma'$ if and only if $u \sim v$ in $\Gamma$. A subset $X$ of $V(\Gamma)$ is said to be \emph{independent} if no two vertices of $X$ are adjacent. The \emph{independence number} of  $\Gamma$  is the cardinality of the largest independent set and it is denoted by $\alpha(\Gamma)$. The \emph{chromatic number} of $\Gamma$, denoted by $\chi(\Gamma)$, is the smallest number of colors needed to color the vertices  of $\Gamma$ so that no two adjacent vertices share the same color. A \emph{clique} in $\Gamma$ is a set of pairwise adjacent vertices. The \emph{clique number} of $\Gamma$ is the size of maximum clique in $\Gamma$ and it is denoted by $\omega(\Gamma)$. It is well known that $\omega(\Gamma) \leq \chi(\Gamma)$ (see \cite{westgraph}). A graph $\Gamma$ is \emph{perfect} if $\omega(\Gamma') = \chi(\Gamma')$ for every induced subgraph $\Gamma'$ of $\Gamma$.
	Recall that the {\em complement} $\overline{\Gamma}$ of $\Gamma$ is a graph with same vertex set as $\Gamma$ and distinct vertices $u, v$ are adjacent in $\overline{\Gamma}$ if they are not adjacent in $\Gamma$. A subgraph $\Gamma'$ of $\Gamma$ is called \emph{hole} if $\Gamma'$  is a  cycle as an induced subgraph, and $\Gamma'$  is called an \emph{antihole} of   $\Gamma$ if   $\overline{\Gamma'}$ is a hole in $\overline{\Gamma}$.

	\begin{theorem}\label{strongperfecttheorem}\cite{strongperfectgraph}
		A finite graph $\Gamma$
		is perfect if and only if it does not contain hole or antihole of odd length at least $5$.
	\end{theorem}
	A subset $D$ of $V(\Gamma)$ is said to be a dominating set if any vertex in $V(\Gamma) \setminus D$ is adjacent to at least one vertex in $D$. If $D$ contains only one vertex then that vertex is called dominating vertex. The \emph{domination number} $\gamma(\Gamma)$ of $\Gamma$ is the minimum size of a dominating set in $\Gamma$. A graph $\Gamma$ is said to be \emph{triangulated} if any vertex is a vertex of a triangle. 
	An edge cover of $\Gamma$ is a subset $L$ of $E(\Gamma)$  such that every vertex of $\Gamma$ is incident to some edge of $L$. The minimum cardinality of an edge cover in $\Gamma$ is called the edge covering number, it is denoted by $\beta'({\Gamma})$. A vertex cover of $\Gamma$ is a subset $Q$ of $V(\Gamma)$ such that it contains at least one endpoint of every edge of $\Gamma$. The minimum cardinality of a vertex cover in $\Gamma$ is called the vertex covering number, it is denoted by $\beta({\Gamma})$. A matching of $\Gamma$  is a set of edges with no share endpoints. The maximum cardinality of a matching in $\Gamma$  is called the matching number of $\Gamma$ and it is denoted by $\alpha'(\Gamma)$. Recall that in a graph $\Gamma$, we have $\alpha(\Gamma) + \beta({\Gamma}) = |V(\Gamma)|$ and for $\Gamma$ with no isolated vertices, we have $\alpha'(\Gamma) + \beta'({\Gamma}) = |V(\Gamma)|$. A graph $\Gamma$ is said to be planar if it can be drawn on a plane without any crossing of its edges. A homomorphism of graph $\Gamma$ to a graph $\Gamma'$  is a mapping $f$ from  $V(\Gamma)$  to $V(\Gamma')$ with the property that if $u \sim v$, then $uf \sim vf$  for all $ u, v \in  V(\Gamma)$. A retraction is a homomorphism $f$ from a graph $\Gamma$  to a subgraph $\Gamma'$  of $\Gamma$ such that $vf = v$ for each vertex $v$ of $\Gamma'$. In this case the subgraph $\Gamma'$ is called a retract of $\Gamma$.         
	

		\section{Inclusion ideal graph of a semigroup}
In this section, we study the algebraic properties of $S$ as well as graph theoretic properties of the inclusion ideal graph $\mathcal{I}n(S)$. First we investigate connectedness of $\mathcal{I}n(S)$. We show that $diam(\mathcal{I}n(S)) \leq 3$ if it is connected. 
Moreover, the clique number, planarity, perfectness and the girth of $\mathcal{I}n(S)$ are investigated.    
	\begin{theorem}\label{disconnected}
		The inclusion ideal graph $\mathcal{I}n(S)$ is disconnected if and only if $S$ contains at least two minimal left ideals and every nontrivial left ideals of $S$ is minimal as well as maximal.
	\end{theorem}
	\begin{proof}
		Suppose that the graph $\mathcal{I}n(S)$ is disconnected. Without loss of generality, we may assume that there exist at least two nontrivial left ideals $I_1,     I_2$ of $S$ such that $I_1 \in C_1$ and $I_2 \in C_2$, where $C_1$ and $C_2$ are two distinct components of $\mathcal{I}n(S)$. Let if possible, $I_1$ is not minimal. Then there exists a nontrivial left ideal $I_k$ of $S$ such that  $I_k \subset I_1$. Now, we have the following cases.
		
		\noindent\textbf{Case 1.} $I_k \cup I_2 \neq S $. Then we have $I_1 \sim I_k \sim (I_k \cup I_2) \sim I_2$, a contradiction.
		
		\noindent\textbf{Case 2.} $I_k \cup I_2 = S $. Then clearly  $I_1 \cup I_2 = S $. Let $x \in I_1$ but $x \notin I_k $. Thus, $x \in I_2$ so that $x \in I_1  \cap I_2$. We get $I_1 \sim (I_1 \cap I_2) \sim  I_2$,  again a contradiction.  Thus, $I_1$ is minimal. Similarly, we obtain $I_2$ is minimal.
		
		Further, on contrary suppose that $I_1$ is not maximal. Then there exists a nontrivial left ideal $I_k$ of $S$ such that $I_1 \subset I_k$. Now we get a contradiction in the following possible cases.
		
		\noindent\textbf{Case 1.} $I_1 \cup I_2 \neq S$. Then $I_1 \sim (I_1 \cup I_2) \sim  I_2$ gives a contradiction.
		
		\noindent\textbf{Case 2.} $I_1 \cup I_2 = S$. Then clearly $I_k \cup I_2 = S$ so that there exists $x \in I_k \cap I_2$. Thus,  we have  $I_1 \sim I_k \sim (I_k \cap I_2) \sim I_2$, again a contradiction. Hence $I_1$ is maximal. Similarly, one can observe that $I_2$ is maximal.
		
		The converse follows from Remark \ref{disjoint intersection minimal}.
	\end{proof}
	
	\begin{corollary}\label{null graph}
If the graph $\mathcal{I}n(S)$ is disconnected then it is a null graph (i.e. it has no edge). 
\end{corollary}
\begin{lemma}\label{maximal left ideal1}
	If  the semigroup $S = I_1 \cup I_2$, where $I_1$ and $I_2$ are  minimal left ideals then $I_1$ and $I_2$ are maximal left ideals.
\end{lemma}
\begin{proof}
	Suppose that $S = I_1 \cup I_2$ where $I_1$ and  $I_2$ are minimal left ideals of $S$. If there exists another nontrivial left ideal $I_k$ of $S$ then there exists $x \in I_k$. Consequently, either $x \in I_1$ or $x \in I_2$  follows that either $Sx \subset I_1$ or $Sx \subset I_2$. Since $I_1$ and $I_2$ are minimal left ideals of $S$ we get either $I_1 = Sx \subset I_k$ or $I_2  \subset I_k$. If $I_1 \subset I_k$ then there exists $y \in I_k$ such that $y \notin I_1$. This implies that $y \in I_2$ so that $Sy \subset I_2$. Consequently, $I_2 = Sy \subset I_k$. Thus, $I_1 \cup I_2 = S \subset I_k$ so that $I_k = S$, a contradiction. Similarly, one can observe for $I_2 \subset I_k$ we have $I_k = S$, again a contradiction. Therefore, $I_1$ and $I_2$ are maximal left ideals.
	\end{proof}
	\begin{theorem}\label{two minimal}
		The graph $\mathcal{I}n(S)$ is disconnected if and only if $S$ is the union of exactly two minimal left ideals.
	\end{theorem}
	\begin{proof}
		Suppose first that $\mathcal{I}n(S)$ is disconnected. Then by Theorem \ref{disconnected}, each nontrivial left ideal is minimal. Suppose $S$ has at least three minimal left ideals, namely $I_1, I_2$
		and $I_3$. Then $I_1 \cup I_2$ is a nontrivial left ideal of $S$ which is not minimal. Consequently, by Theorem \ref{disconnected}, we get a contradiction of the fact that $\mathcal{I}n(S)$ is disconnected. Thus, $S$ has exactly two minimal left ideals. If $S \neq I_1 \cup I_2$ then $I_1 \cup I_2$ is a nontrivial left ideal which is not minimal, a contradiction ( cf. Theorem \ref{disconnected}). Thus, $S = I_1 \cup I_2$. Converse part follows from   Theorem \ref{disconnected} and Lemma \ref{maximal left ideal1}.
	\end{proof}
	
	\begin{theorem}\label{diameter3}
		If $\mathcal{I}n(S)$ is a connected graph then $diam(\mathcal{I}n(S))$ $\leq$ $3$.	
	\end{theorem}
	\begin{proof}
		Let $I_1, I_2$ be two nontrivial left ideals of $S$. If $I_1 \sim I_2$ then $d(I_1, I_2)$ = 1. If $I_1 \nsim I_2$ then in the following cases we show that $d(I_1, I_2) $$\leq 3$.
		
		\noindent\textbf{Case 1.} $I_1 \cup I_2 \neq S$. Then $I_1 \sim (I_1 \cup I_2) \sim I_2$ so that $d(I_1, I_2)$ = 2.
		
		\noindent\textbf{Case 2.} $I_1 \cup I_2 = S$. If $I_1 \cap I_2 \neq \emptyset$ then $I_1 \sim (I_1 \cap I_2) \sim I_2$  gives  $d(I_1, I_2) $= 2. We may now suppose that $I_1 \cap I_2 = \emptyset$. Since $\mathcal{I}n(S)$ is a connected graph, there exists a nontrivial left ideal $I_k$ of $S$ such that either $I_k \subset I_1$ or $I_1 \subset I_k$.  If  $I_k \subset I_1$ then there exists $x \in I_1$ but $x \notin I_k $. If $I_k \cup I_2 = S$ then $x \in I_2$. Thus, we  get $I_1 \cap I_2 \neq \emptyset$, a contradiction. Consequently, $I_2 \cup I_k \neq S$. Further, we get a path $I_1 \sim I_k \sim (I_2 \cup I_k) \sim I_2$ of length 3. Thus, $d(I_1, I_2)$ = 3. Now if $I_1 \subset I_k$ then note that  $I_k \cap I_2 \neq \emptyset$. Consequently, we get $I_1 \sim I_k \sim (I_k \cap I_2) \sim I_2$  between $I_1$ and $I_2$. Hence, $diam(\mathcal{I}n(S))$ $\leq$ $3$.
	\end{proof}
	\begin{lemma}\label{same union not adjacent}
		Let $\mathcal{I}$ and $\mathcal{I}'$ be two distinct left ideals of $S$ such that  both  $\mathcal{I}$ and $\mathcal{I}'$ are the union of $k$ minimal left ideals of $S$. Then $\mathcal{I} \nsim \mathcal{I}'$ in  $\mathcal{I}n(S)$.
	\end{lemma}
	\begin{proof}
		On contrary suppose that $\mathcal{I} \sim \mathcal{I}'$. Without loss of generality assume that $\mathcal{I} \subset \mathcal{I}'$. Since $\mathcal{I}$ and $\mathcal{I}'$ are the union of $k$ minimal left ideals of $S$ and $\mathcal{I} \subset \mathcal{I}'$, by Remark \ref{disjoint intersection minimal} we get $\mathcal{I} = \mathcal{I}'$, a contradiction.
	\end{proof}

Let us denote ${\rm Min}(S)$ by the set of all minimal left ideals of $S$. By a nontrivial left ideal $I_{{i_1}{i_2}\cdots{i_k}}$, we mean $I_{i_1} \cup I_{i_2} \cup \cdots \cup I_{i_k}$, where $I_{i_1}, I_{i_2}, \cdots, I_{i_k}$ are minimal left ideals of $S$. 	

	\begin{lemma}\label{length4 or 5}
		If $\mathcal{I}n(S)$ has a cycle of length $4$ or $5$ then $\mathcal{I}n(S)$ has a triangle.
	\end{lemma}
	\begin{proof}
		Suppose first that $\mathcal{I}n(S)$ has a cycle of length 5 such that $C : I_1 \sim I_2 \sim I_3 \sim I_4 \sim I_5 \sim I_1$. From the adjacency of ideals in $C$, note that there exists a chain $I_i \subset I_j \subset I_k$ in $S$. Thus $\mathcal{I}n(S)$ has a triangle.
		
		Now we suppose that there exists a cycle $C : I_1 \sim I_2 \sim I_3 \sim I_4 \sim I_1$  of length 4 in $\mathcal{I}n(S)$. Assume that $I_1 \nsim I_3$ and $I_2 \nsim I_4$. Since $I_1 \sim I_2$ we have either $I_1 \subset I_2$ or $I_2 \subset I_1$. If $I_1 \subset I_2$ then $I_3 \subset I_2$ and $I_3 \subset I_4$. It follows that $I_1 \cup I_3 \subseteq I_2$ and $I_3 \subseteq I_2 \cap I_4$. Consequently, we obtain either $I_2 \sim (I_1 \cup I_3) \sim I_3 \sim I_2$ or $I_2 \sim (I_2 \cap I_4) \sim I_3 \sim I_2$. Thus, $\mathcal{I}n(S)$ has a triangle. If $I_2 \subset I_1$ then $I_2 \subset I_3$ and $I_4 \subset I_3$. It follows that $I_2 \subseteq I_1 \cap I_3$ and  $I_2 \cup I_4 \subseteq I_3$. Further, we get a cycle either $I_2 \sim I_1 \sim (I_1 \cap I_3) \sim I_2$ or $I_2 \sim (I_2 \cup I_4) \sim I_3 \sim I_2$. Hence, we have the result.
	\end{proof}
	
	In the following theorem we determine the girth of $\mathcal{I}n(S)$.
\begin{theorem}
For a semigroup $S$, we have $g(\mathcal{I}n(S)) \in \{3, 6, \infty \}$.
\end{theorem}
\begin{proof}
If $\mathcal{I}n(S)$ is disconnected or a tree, then clearly $g(\mathcal{I}n(S)) = \infty$. Suppose that $S$ has $n$ minimal left ideals. Now we prove the result through following cases.

\noindent\textbf{Case 1.} $n = 0$. If $S$ has no nontrivial left ideal then there is nothing to prove. Otherwise there exists a chain of nontrivial left ideals of $S$ such that $I_1 \supset I_2 \supset \cdots \supset I_k \supset \cdots$. Thus, $g(\mathcal{I}n(S)) = 3$.

\noindent\textbf{Case 2.} $n = 1$. Suppose that $I_1$ is the only minimal left ideal of $S$. Since $I_1$ is unique minimal left ideal then it is contained in all other nontrivial left ideals of $S$. If any two non minimal left ideals are adjacent, then $g(\mathcal{I}n(S)) = 3$. Otherwise, being a star graph, $g(\mathcal{I}n(S)) = \infty$.

\noindent\textbf{Case 3.} $n = 2$. Let $I_1, I_2$ be two minimal left ideals of $S$. If $I_1 \cup I_2 = S$ then by Theorem \ref{disconnected} and Corollary \ref{null graph}, $g(\mathcal{I}n(S)) = \infty$. If $I_1 \cup I_2 \neq S$, then $J = I_1 \cup I_2$ is a nontrivial left ideal of $S$. If $S$ has only these three, namely  $I_1, I_2$ and $J$, left ideals then we obtain $I_1 \sim J \sim I_2$. Then $g(\mathcal{I}n(S)) = \infty$. Now suppose that  $S$ has a nontrivial left ideal $K$ other than $I_1, I_2$ and $J$. Now we have the following subcases.

\textbf{Subcase 1.} Both $I_1, I_2$ contained in $K$. Then $I_1 \cup I_2 = J \subset K$. Consequently, $I_1 \sim J \sim K \sim I_1$ so that  $g(\mathcal{I}n(S)) = 3$.

\textbf{Subcase 2.} Either $I_1 \subset K$ or $I_2 \subset K$. Without loss of generality, assume that $I_1 \subset K$. If there exists a nontrivial left ideal $K'$ of $S$ such that $K \subset K'$ then $I_1 \sim K \sim K' \sim I_1$ follows that $g(\mathcal{I}n(S)) = 3$. Otherwise, $K$ is a maximal left ideal of $S$. Consequently, $I_2 \cup K = S$. If $J$ is not maximal, then by Subcase 1 we get $g(\mathcal{I}n(S)) = 3$. We may now suppose that $J$ is also a maximal left ideal of $S$. Now we claim that $V(\mathcal{I}n(S)) = \{I_1, I_2, J, K \}$. Let if possible $J' \in V(\mathcal{I}n(S))$ but $J' \notin \{I_1, I_2, J, K \}$. Since $J$ is maximal, for any $a \in K \setminus I_1$, we have $Sa = K$. Since $J'$ is a nontrivial left ideal of $S$, there exists an element $b \in J'$ so that $b$ is either in $I_2$ or in $K$. If $b \notin I_2$, then $b \in K$. Consequently, $K \subset J'$, a contradiction to the maximality of $K$. Now suppose that $b \in I_2$. By the minimality of $I_2$, we obtain $I_2 \subset J'$. Since $I_2 \neq J'$, there exists an element $c \in J'$ but $c \notin I_2$. Consequently, $I_2 \cup K = S$. Thus $c \in K$ so that $K \subset J'$. It follows that $I_2 \cup K = S \subset J'$ and so $J' = S$. Thus, we get  $V(\mathcal{I}n(S)) = \{I_1, I_2, J, K \}$ such that $K \sim I_1 \sim J \sim I_2$. Therefore, $g(\mathcal{I}n(S)) = \infty$. 

\noindent\textbf{Case 4.} $n =3$. Let $I_1, I_2, I_3$ be the minimal left ideals of $S$. If $I_1 \cup I_2 \cup I_3 \neq S$, then we have $I_1 \subset (I_1 \cup I_2) \subset (I_1 \cup I_2 \cup I_3)$. It follows that, $g(\mathcal{I}n(S)) = 3$. Further, suppose that $I_1 \cup I_2 \cup I_3 = S$. Then  all the nontrivial left ideals of $S$ are $I_1, I_2, I_3, I_1 \cup I_2, I_1 \cup I_3$ and $I_2 \cup I_3$. Infact $\mathcal{I}n(S) \cong C_6$. Thus, $g(\mathcal{I}n(S)) = 6$. 

\noindent\textbf{Case 5.}  $n \geq 4$. Then for minimal left ideals $I_1, I_2$ and $I_3$, we get a triangle  
 $I_1 \sim (I_1 \cup I_2) \sim (I_1 \cup I_2 \cup I_3) \sim I_1$ so that  $g(\mathcal{I}n(S)) = 3$.
 
Hence, from above cases we have  $g(\mathcal{I}n(S)) \in \{3, 6, \infty \}$.
\end{proof}

	\begin{theorem}\label{perfect theorem finite no of left ideals}
		Let $S$ be a semigroup with finite number of left ideals. Then the graph $\mathcal{I}n(S)$ is perfect.
	\end{theorem}
	\begin{proof}
		In view of Theorem \ref{strongperfecttheorem}, we show that $\mathcal{I}n(S)$ does not contain a hole or an antihole of odd length at least five. On contrary, assume that $\mathcal{I}n(S)$ contains a hole $C: I_1 \sim I_2 \sim I_3 \sim \cdots \sim I_{2n+1} \sim I_1$, where $n \geq 2$. Since $I_1 \sim I_2$, we have either $I_1 \subset I_2$ or $I_2 \subset I_1$. Without loss of generality, suppose that $I_1 \subset I_2$. Then clearly $I_3 \subset I_2$. Otherwise $I_1 \sim I_3$, a contradiction. Further for $2 \leq i \leq n$ note that $I_{2i-1} \subset I_{2i}$. Since $I_{2n} \sim I_{2n+1}$, we have either $I_{2n} \subset I_{2n+1}$ or $I_{2n+1} \subset I_{2n}$. But $I_{2n} \subset I_{2n+1}$ is not possible because $I_{2n-1} \subset I_{2n}$ follows that $I_{2n-1} \sim I_{2n+1}$, a contradiction. Also $I_{2n+1} \sim I_1$ will give $I_{2n+1} \subset I_1$. Consequently, from $I_1 \subset I_2$ we obtain $I_{2n+1} \subset I_2$. Thus, $I_{2n+1} \sim I_2$, a contradiction.
		
		Next, suppose that $\mathcal{I}n(S)$ contains an antihole $C$ of length at least five. Then $\overline{C}$: $I_1 \sim I_2 \sim I_3 \sim \cdots \sim  I_{2n+1} \sim I_1$, where $n \geq 2$, is a hole in $\overline{\mathcal{I}n(S)}$. Since $I_1 \sim I_3$ in $\mathcal{I}n(S)$, we have either $I_1 \subset I_3$ or $I_3 \subset I_1$. Without loss of generality, assume that $I_1 \subset I_3$. Then, for $4 \leq j \leq 2n$, note that $I_1 \subset I_j$. Moreover, $I_2 \subset I_j$ for $4 \leq j \leq 2n+1$. Since $I_3 \sim I_5$ in $\mathcal{I}n(S)$ we have either $I_3 \subset  I_5$ or $I_5 \subset I_3$. If $I_5 \subset I_3$ then $I_2 \subset I_3$ as $I_2 \subset I_5$. Thus, $I_2 \sim I_3$ in $\mathcal{I}n(S)$ which is not possible. Consequently, $I_3 \subset I_5$. Also, it is easy to observe that $I_3 \subset I_{2n+1}$. Since $I_1 \subset I_3$ we have $I_1 \subset I_{2n+1}$ so that $I_1 \sim I_{2n+1}$ in $\mathcal{I}n(S)$, a contradiction. Thus, $\mathcal{I}n(S)$ does not contain an antihole of length at least five.
	\end{proof}
	
	Let $S$ be a semigroup with $n$ minimal left ideals. Now we classify the semigroups for which $\omega(\mathcal{I}n(S))$ is $n$. 
	\begin{lemma}\label{cliqueinclusion}
Let $S$ be a semigroup such that $S = I_{{i_1}{i_2}\cdots{i_n}}$. Then $\omega(\mathcal{I}n(S)) = n-1$.
\end{lemma}
\begin{proof}
Since we have $n$ minimal left ideals, namely $I_{i_1}, I_{i_2}, \ldots, I_{i_n}$. Note that $\mathcal{C} = \{I_{i_1}, I_{{i_1}{i_2}}, \ldots, I_{{i_1}{i_2}\cdots {i_{n-1}}}\}$ be a clique of size $n-1$ in $\mathcal{I}n(S)$. Let $\mathcal{C} \cup \{J\}$ be a clique in $\mathcal{I}n(S)$, where $J= I_{{i_1}{i_2}\cdots{i_{k}}}$ for some $k$, $1 \leq k \leq n-1$. Then $J$ is adjacent with every element of $\mathcal{C}$, a contradiction (see Lemma \ref{same union not adjacent}). Consequently, $\mathcal{C}$ is a maximal clique in $\mathcal{I}n(S)$. Now if $\mathcal{C}'$ be a clique of size $n$ then there exist two nontrivial left ideals in $\mathcal{C}'$ which are union of $k$ minimal left ideals for some $k$, where $1 \leq k \leq n-1$. By Lemma \ref{same union not adjacent}, a contradiction for $\mathcal{C}'$ to be a clique in $\mathcal{I}n(S)$. Hence, $\omega(\mathcal{I}n(S)) = n-1$.
\end{proof}
\begin{theorem}
Let $S$ be a semigroup with $n$ minimal left ideals. Then $\omega(\mathcal{I}n(S)) = n$ if and only if $I_{{i_1}{i_2}\cdots{i_n}}$ is a maximal left ideal. 
\end{theorem}
\begin{proof}
Suppose that $\omega(\mathcal{I}n(S)) = n$. Clearly, by Lemma \ref{cliqueinclusion}, we have $S \neq I_{{i_1}{i_2}\cdots{i_n}}$.
 Let if possible, $I_{{i_1}{i_2}\cdots{i_n}}$ is not a maximal left ideal of $S$. Then there exists a nontrivial left ideal $K$ of $S$ such that $I_{{i_1}{i_2}\cdots{i_n}} \subset K$. Note that  $\mathcal{C} = \{I_{i_1}, I_{{i_1}{i_2}}, \ldots, I_{{i_1}{i_2}\cdots {i_{n-1}}}, I_{{i_1}{i_2}\cdots {i_{n}}}, K\}$ forms a clique of size $n+1$. Consequently, $\omega(\mathcal{I}n(S)) \neq n$, a contradiction. Thus, $I_{{i_1}{i_2}\cdots {i_{n}}}$ is a maximal left ideal of $S$.

Conversely, suppose that $I_{{i_1}{i_2}\cdots {i_{n}}}$ is a maximal left ideal of $S$. Then by Lemma \ref{S minus K is lclass}, $S \setminus I_{{i_1}{i_2}\cdots {i_{n}}}$ is an $\mathcal{L}$-class. Thus, for each $a \in S \setminus I_{{i_1}{i_2}\cdots {i_{n}}}$, we get either $S^{1}a = S$ or $S^{1}a$ is a nontrivial left ideal of $S$. First suppose that $S^{1}a = S$. Therefore, for any nontrivial left ideal $I$ of $S$, note that if for some $a \in S \setminus I_{{i_1}{i_2}\cdots {i_{n}}}$ such that $a \in I$, then $I = S$, a contradiction. Thus, every nontrivial left ideal $I$ of $S$ is either  a minimal left ideal or a  union of minimal left ideals. Consequently, in the similar lines of the proof of Lemma \ref{cliqueinclusion}, we get a clique $\mathcal{C} = \{I_{i_1}, I_{{i_1}{i_2}}, \ldots, I_{{i_1}{i_2}\cdots {i_{n-1}}}, I_{{i_1}{i_2}\cdots {i_{n}}}\}$ of maximum size $n$ so that $\omega(\mathcal{I}n(S)) = n$.  

We may now suppose that for each $a \in S \setminus I_{{i_1}{i_2}\cdots {i_{n}}}$, $S^{1}a = I$ is a nontrivial left ideal of $S$. Assume that $J$ is any nontrivial left ideal of $S$ such that $a \in J$ for some $a \in S \setminus I_{{i_1}{i_2}\cdots {i_{n}}}$. Then $S \setminus I_{{i_1}{i_2}\cdots {i_{n}}} \subset I \subseteq J$. Consequently, all the vertices of $\mathcal{I}n(S)$ are minimal left ideals, union of minimal left ideals and of the form $(S \setminus I_{{i_1}{i_2}\cdots {i_{n}}}) \cup I_{{j_1}{j_2}\cdots{j_k}}$. Note that $I_{i_1} \sim (S \setminus I_{{i_1}{i_2}\cdots {i_{n}}}) \cup I_{{j_1}{j_2}\cdots{j_k}}$ implies $i_1 \in \{j_1, j_2, \ldots, j_k \}$ and $(S \setminus I_{{i_1}{i_2}\cdots {i_{n}}}) \cup I_{{s_1}{s_2}\cdots{s_k}} \sim (S \setminus I_{{i_1}{i_2}\cdots {i_{n}}}) \cup I_{{t_1}{t_2}\cdots{t_p}}$ implies $I_{{s_1}{s_2}\cdots{s_k}} \sim I_{{t_1}{t_2}\cdots{t_p}}$.
Suppose that $K$ is of the form $(S \setminus I_{{i_1}{i_2}\cdots {i_{n}}}) \cup I_{{j_1}{ji_2}\cdots{j_k}}$. Let $\mathcal{C}$ be an arbitrary clique such that $K \in \mathcal{C}$. Note that $\mathcal{C} = \{I_{j_1}, I_{{j_1}{j_2}}, \ldots, I_{{j_1}{j_2}\cdots {j_{k}}}, K, K \cup I_{j_{k+1}}, \ldots, K \cup I_{{j_{k+1}}{j_{k+2}}\cdots{j_{n-k-1}}} \}$ is a clique of size $n$. If $\mathcal{C} \cup \{K'\}$ is a clique of size $n+1$ then either $K' = I_{{j_1}{j_2}\cdots{j_t}}$ or $K' = (S \setminus I_{{i_1}{i_2}\cdots {i_{n}}}) \cup I_{{j_1}{j_2}\cdots{j_s}}$. Then $K'$ is not adjacent with at least one vertex of $\mathcal{C}$, a contradiction of the fact that $\mathcal{C} \cup \{K'\}$ is a clique. Consequently, $\mathcal{C}$ is a maximal clique in $\mathcal{I}n(S)$. 
Further, suppose that $\mathcal{C}$ do not contain any nontrivial left ideal of the form $(S \setminus I_{{i_1}{i_2}\cdots {i_{n}}}) \cup I_{{j_1}{j_2}\cdots{j_k}}$. Note that $\mathcal{C} = \{I_{i_1}, I_{{i_1}{i_2}}, \ldots, I_{{i_1}{i_2}\cdots {i_{n-1}}}, I_{{i_1}{i_2}\cdots {i_{n}}}\}$ is a maximal clique of size $n$. Now suppose that $\mathcal{C}'$ is an arbitrary clique of size at least $n+1$. Then by the adjacency of vertices in $\mathcal{I}n(S)$ mentioned above and in Lemma \ref{same union not adjacent}, there exist at least two vertices $U$ and $U'$ such that $U \nsim U'$. Thus, $\omega(\mathcal{I}n(S)) = n$ and the proof is complete.
\end{proof}

\begin{theorem}
For the graph $\mathcal{I}n(S)$, we have the following results:
\begin{enumerate}
\item[{\rm(i)}] If $\mathcal{I}n(S)$ is a planar graph then $| {\rm Min}(S) | \leq 4$.
\item [{\rm(ii)}] Let $S$ be the union of $n$ minimal left ideals. Then  $\mathcal{I}n(S)$ is a planar graph if and only if $n \leq 4$.  
\end{enumerate}
\end{theorem}
\begin{proof}
		(i) Suppose that $| {\rm Min}(S) | =5$ with ${\rm Min}(S) = \{I_1, I_2, I_3, I_4, I_5\}$. Then, from the graph given below, note that $\mathcal{I}n(S)$ contains a subdivision of complete bipartite $K_{3,3}$ as a subgraph.
		
		\begin{figure}[h!]
			\centering
			\includegraphics[width=0.4\textwidth]{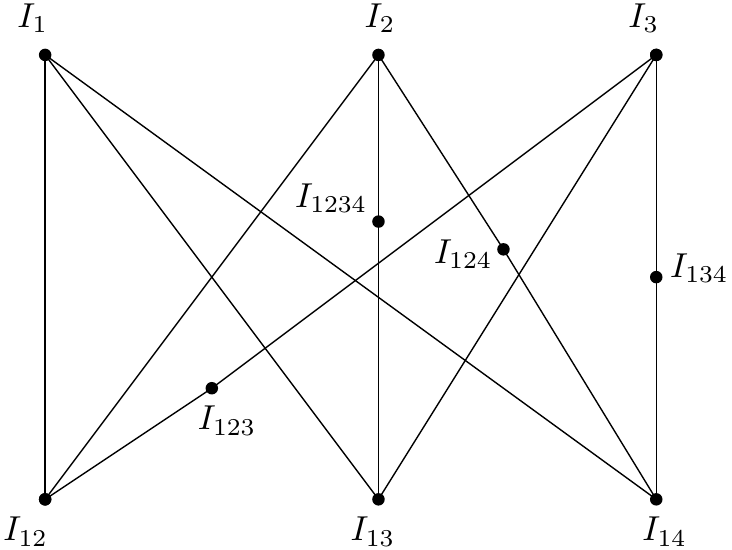}
			\caption{Subgraph  of  $\mathcal{I}n(S)$  homeomorphic to $K_{3,3}$}
		\end{figure}
	
		For $| {\rm Min}(S) | \geq 6$, note that $I_1 \subset I_{12} \subset I_{123} \subset I_{1234} \subset I_{12345}$ be a chain of nontrivial left ideals of $S$. Consequently, $\mathcal{I}n(S)$ contains a subgraph isomorphic to $K_5$. Thus, by Kurwatowski theorem, $\mathcal{I}n(S)$ is nonplanar.
		
		(ii) The proof for $\mathcal{I}n(S)$ is nonplanar for $n \geq 5$ follows  from part (i).
		By Corollary \ref{null graph} and Theorem \ref{two minimal}, $\mathcal{I}n(S)$ is planar for $n= 2$. For $n= 3$, note that $\mathcal{I}n(S) \cong C_6$ so that $\mathcal{I}n(S)$ is planar. The planarity of $\mathcal{I}n(S)$ can be seen from the following graph of $\mathcal{I}n(S)$, for $n = 4$.
		\begin{figure}[h!]
			\centering
			\includegraphics[width=0.5 \textwidth]{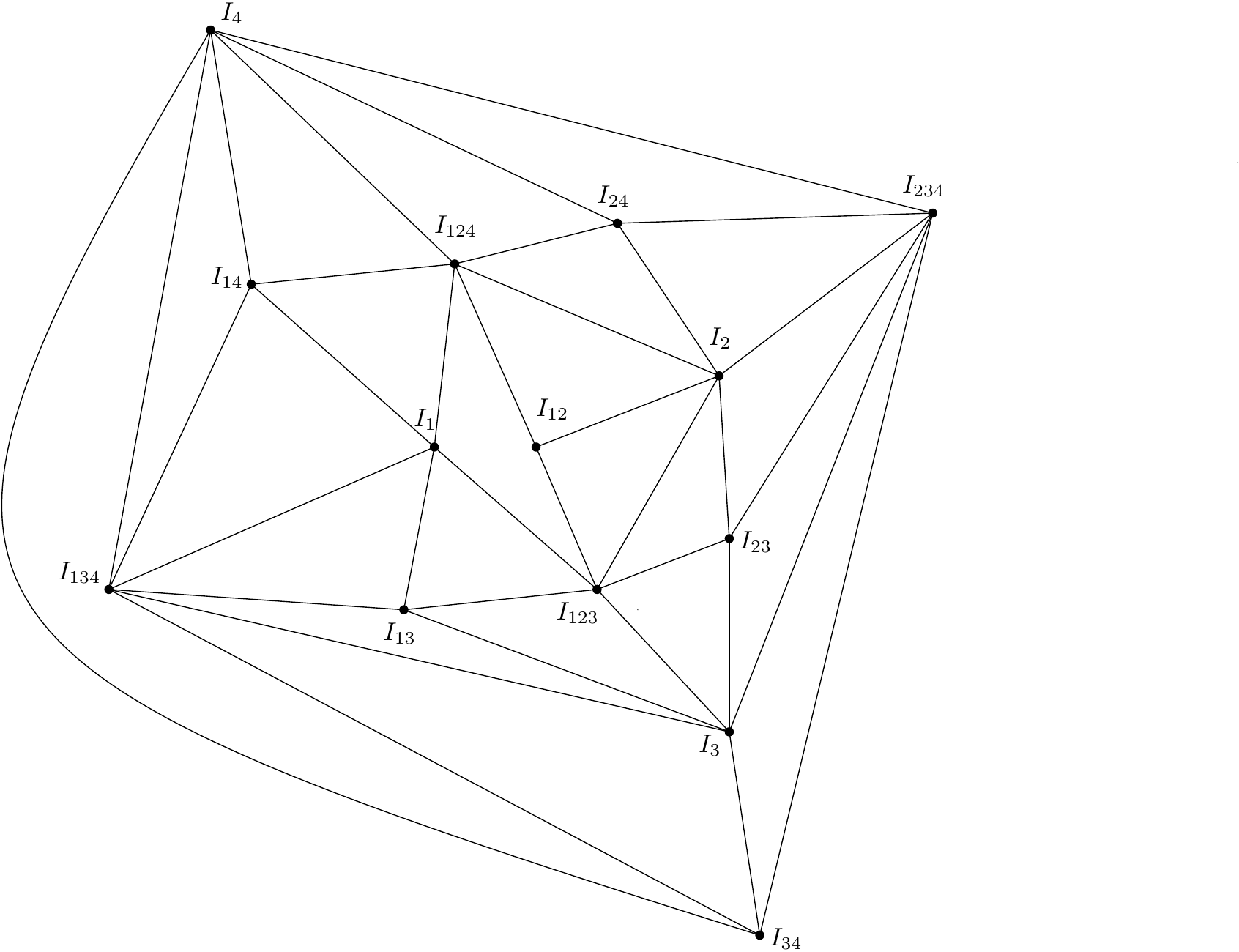}
			\caption{Planar drawing of  $\mathcal{I}n(S)$}
		\end{figure}
		
		\end{proof}
\newpage
\section{Inclusion ideal graph of completely simple semigroup}

In this section, we study various graph invariants including the dominance number, clique number, independence number of the inclusion ideal graph of a completely simple semigroup $S$. We also prove that the graph $\mathcal{I}n(S)$ has a perfect matching (cf. Theorem \ref{perfectmatchinginclusion}). In what follows, for $n \in \mathbb{N}$, we denote $[n] = \{\ 1, 2, \ldots, n \}$. For a completely simple semigroup having $n$ minimal left ideals, we write a nontrivial left ideal $I_{{i_1}{i_2} \cdots {i_{k}}} = I_{i_1} \cup  I_{i_2} \cup  I_{i_3} \cup \cdots \cup  I_{i_k}$ such that $i_1, i_2, \ldots, i_k \in [n]$ and $1 \leq k \leq n-1$, where $I_{i_1}, I_{i_2}, \ldots, I_{i_k}$ are minimal left ideals of $S$.

\begin{lemma}
Let $S$ be a completely simple semigroup with $n$ minimal left ideals. Then $\mathcal{I}n(S)$ is disconnected for $n = 2$, and connected for $n \geq 3$. Moreover, if $\mathcal{I}n(S)$ is connected then $diam(\mathcal{I}n(S)) = 3$.
	\end{lemma}
	\begin{proof}
	 By Lemma \ref{union minimal} and  Theorem \ref{two minimal}, $\mathcal{I}n(S)$ is disconnected for $n = 2$. For $n \geq 3$, as a consequence of Theorem \ref{two minimal}, $\mathcal{I}n(S)$ is connected. Let $I_1, I_{{2}{3}{4}\cdots{n}}$ be two nontrivial left ideals. Then there exists a shortest path $I_1 \sim I_{{1}{2}} \sim I_2 \sim I_{{2}{3}{4}\cdots{n}}$ such that $d(I_1, I_{{2}{3}{4}\cdots{n}}) = 3$. By Theorem \ref{diameter3}, $diam(\mathcal{I}n(S)) = 3$.
	\end{proof}
	\begin{theorem}\label{order}
		Let $S$ be a completely simple semigroup with $n$ minimal left ideals. Then $\mathcal{I}n(S)$ is a graph of order $2^n - 2$.
	\end{theorem}
	\begin{proof}
		In view of Lemma \ref{union minimal}, the vertices of $\mathcal{I}n(S)$ are either minimal left ideals or union of minimal left ideals. In addition to $n$ minimal left ideals, we have $n \choose 2$ + $n \choose 3$+ $\cdots $+$n \choose n-1$ nontrivial left ideals as a union of minimal left ideals. Thus, we obtain $n \choose 1$ +$n \choose 2$ + $n \choose 3$+ $\cdots $+$n \choose n-1$ $= 2^n-2$ nontrivial left ideals of $S$. Hence, $| V(\mathcal{I}n(S))| = 2^n-2$.
	\end{proof}
	Now in the following lemma we obtain the degree of each vertex of $\mathcal{I}n(S)$. 
	
	\begin{lemma}\label{degree}
		Let $S$ be a completely simple semigroup with $n$ minimal left ideals and let $K = I_{{i_1}{i_2} \cdots {i_{k}}}$ be a  nontrivial left ideal of $S$. Then $deg(K) = (2^k-2) + (2^{n-k}-2)$.
	\end{lemma}
	\begin{proof}
		Let $K = I_{{i_1}{i_2} \cdots {i_{k}}}$ be a nontrivial left ideal of $S$. Then nontrivial left ideals of $S$ which are adjacent to $K$ are  nontrivial left ideals properly contained in $K$ and nontrivial left ideals of $S$ properly containing $K$. By the proof of Theorem \ref{order}, we have $k \choose 1$ +$k \choose 2$ + $k \choose 3$+ $\cdots $+$k \choose k-1$ $= 2^k-2$ nontrivial left ideals which are properly contained in $K$. Further, if $K (=I_{{i_1}{i_2} \cdots {i_{k}}}) \subset J$ then note that $J = I_{{i_1}{i_2} \cdots {i_{k}}{i_{k+1}}{i_{k+2}} \cdots {i_{s}}}$ such that $i_{k+1}, i_{k+2}, \ldots, i_{s} \in [n] \setminus \{i_1, i_2, \ldots, i_k\}$ and $1 \leq s \leq n-k-1$. Consequently, we have $\sum_{i=1}^{n-k-1} \binom{n-k}{i} = 2^{n-k}-2$ nontrivial left ideals which properly contains $K$. Hence, $deg(K) = (2^k-2) + (2^{n-k}-2)$.
	\end{proof}
	
	\begin{corollary}
		Let $S$ be a completely simple semigroup with $n$ minimal left ideals. Then the graph $\mathcal{I}n(S)$ is Eulerian for $n  \geq 3$.
	\end{corollary}
\begin{theorem}	Let $S$ be a completely simple semigroup with $n$ minimal left ideals. Then

\[g(\mathcal{I}n(S)) =
  \begin{cases}
\infty &  $n = 2$ \\
6 &  $n = 3$\\
3 & \text{otherwise}
  \end{cases}\]
\end{theorem}
	\begin{proof}
	If $n = 2$ then by Theorem \ref{two minimal} and Corollary \ref{null graph}, the graph $\mathcal{I}n(S)$ is disconnected. It follows that $g(\mathcal{I}n(S)) = \infty$. If $n =3$ then by Lemma \ref{union minimal}, $\mathcal{I}n(S) \cong C_6$. Consequently, $g(\mathcal{I}n(S)) = 6$. For $n \geq 4$, we have at least $I_1, I_2, I_3, I_4$ minimal left ideals of $S$ so that we obtain a cycle $I_1 \sim (I_1 \cup I_2) \sim (I_1 \cup I_2 \cup I_3) \sim I_1$ of length $3$. Thus, $g(\mathcal{I}n(S)) = 3$.
	\end{proof}
\begin{theorem}\label{clique number}
Let $S$ be a completely simple semigroup with $n$ minimal left ideals. Then
\begin{enumerate}
    \item[{\rm(i)}] $\mathcal{I}n(S)$ is a bipartite graph if and only if $n = 3$.
    \item[{\rm(ii)}] the dominance number of $\mathcal{I}n(S)$ is $2$.
    \item[{\rm(iii)}] for $n \geq 4$, $\mathcal{I}n(S)$ is triangulated.
    \item[{\rm(iv)}] the clique number of $\mathcal{I}n(S)$ is $n-1$.
\end{enumerate}
\end{theorem}
\begin{proof}(i)
		If $n= 3$ then by Lemma \ref{union minimal}, $\mathcal{I}n(S) \cong C_6$ which is a bipartite graph.
		Conversely, suppose that $\mathcal{I}n(S)$ is a bipartite graph. Let if possible, $n = 2$. Then by Lemma \ref{union minimal} and Theorem \ref{two minimal}, $\mathcal{I}n(S)$  is disconnected, a contradiction for $\mathcal{I}n(S)$  to be bipartite. Suppose $n \geq 4$ and $I_1,  I_2, I_3$ are the minimal left ideals of $S$. Since $I_1 \subset I_{12} \subset I_{123}$ we get a cycle $I_1 \sim I_{12} \sim I_{123} \sim I_1$ of odd length. Thus, $\mathcal{I}n(S)$ is not a bipartite graph, a contradiction.
		
		(ii) Since there is no dominating vertex in $\mathcal{I}n(S)$, we have $\gamma (\mathcal{I}n(S)) \geq 2$. To prove the result we show that there exists a dominating set of size two in $\mathcal{I}n(S)$. We claim that the set $D = \{I_1, I_{234 \cdots n}\}$ is a dominating set. Let $J = I_{{i_1}{i_2} \cdots {i_{k}}}$, where $i_1, i_2, \ldots, i_k \in [n]$ and $1 \leq  k \leq n-1$ be a nontrivial left ideal of $S$ such that  $J \in  V(\mathcal{I}n(S)) \setminus D$. If some $i_s = 1$ then $I_1 \sim J$. Otherwise, for $1 \leq k \leq n-2$, $J \sim I_{234 \cdots n}$. Thus, $D$ is a dominating set of size two.
		
		(iii) We show that any vertex of $\mathcal{I}n(S)$ is a vertex of a triangle. Let $J = I_{{i_1}{i_2} \cdots {i_{k}}}$, where $i_1, i_2, \ldots, i_k \in [n]$ and $1 \leq  k < n$. If $k=1$ then $J = I_{i_1} \sim I_{{i_1}{i_2}} \sim I_{{i_1}{i_2}{i_3}} \sim J$ gives a triangle. If $k =2$, then we have $J = I_{{i_1}{i_2}} \sim I_{{i_1}{i_2}{i_3}} \sim I_{i_1} \sim J$. Consequently, we get a triangle. If $k \geq 3$ then note that $J = I_{{i_1}{i_2}{i_3}\cdots {i_k}} \sim I_{i_1} \sim  I_{{i_1}{i_2}} \sim J$ is a triangle. Hence,  $\mathcal{I}n(S)$ is triangulated.
		
		(iv) The result follows from Lemma \ref{cliqueinclusion}. 
		\end{proof}
	
	In view of Theorem \ref{perfect theorem finite no of left ideals}, we have the following corollary of Theorem \ref{clique number}.
	\begin{corollary}
		Let $S$ be a completely simple semigroup with $n$ minimal left ideals. Then $\chi(\mathcal{I}n(S))  = n-1$.
	\end{corollary}
	
\begin{theorem}
		Let $S$ be a completely simple semigroup with $n$ minimal left ideals. Then the graph  $\mathcal{I}n(S)$ is edge transitive  if and only if  $n \in \{2,3\}$.
	\end{theorem}
	\begin{proof}
	It is well known that  edge transitive graphs are either vertex transitive or bipartite. For $n \geq 4$, by Lemma \ref{degree}, $\mathcal{I}n(S)$ is not a regular graph so is not  vertex transitive. Also for $n \geq 4$, $g(\mathcal{I}n(S)) = 3$, hence $\mathcal{I}n(S)$ is not a bipartite graph. Thus,  $\mathcal{I}n(S)$ is not an edge transitive graph. Conversely, suppose that $n \in \{2, 3\}$. If $n =2$, then by Corollary \ref{null graph} and Theorem \ref{two minimal}, $\mathcal{I}n(S)$ is edge transitive. By Lemma \ref{union minimal}, $\mathcal{I}n(S) \cong C_6$, for $n = 3$, which is an edge transitive graph.
	\end{proof}
	Now we determine the independence number of the graph $\mathcal{I}n(S)$.
	\begin{remark}\label{binomial inequality}
		It is well known that
	\begin{itemize}
	    \item[(i)] $n \choose 1$ $\leq$ $n \choose 2$ $\leq$ $n \choose 3$ $\cdots$ $\leq$ $n \choose p$ $\geq$ $n \choose p+1$ $\geq$ $n \choose p+2$ $\cdots$ $\geq$ $n \choose n-1$, if $n = 2p$
		\vspace{.3cm}
		\item[(ii)] $n \choose 1$ $\leq$ $n \choose 2$ $\leq$ $n \choose 3$ $\cdots$ $\leq$ $n \choose p$ $=$ $n \choose p+1$ $\geq$ $n \choose p+2$ $\cdots$ $\geq$ $n \choose n-1$, if $n = 2p+1$.
	
	\end{itemize}
\end{remark}
	\begin{lemma}[{ \cite[Theorem 3.1.11]{westgraph}}](Hall's theorem)\label{hall theorem}
		Let $\Gamma$ be a finite bipartite graph with bipartite sets $X$ and $Y$. For a set $X'$ of   vertices in $X$, let $N_{\Gamma}$($X'$) denote the neighbourhood of $X'$ in $\Gamma$, i.e. the set of all vertices in $Y$ adjacent to some elements of $X'$. There is a matching that entirely covers $X$ if and only if every subset $X'$ of $X$ : $| X' |  \leq  |N_{\Gamma}(X')|$. 	
	\end{lemma}
		
		Define $T_k = \{I_{{i_1}{i_2} \cdots {i_{k}}} :  i_1, i_2, \ldots, i_k \in [n]\}$  and  $M_k$  (where $k = 1, 2, \ldots, n-2$) be the induced bipartite subgraph of $\mathcal{I}n(S)$ with vertex set $T_k$ and   $T_{k+1}$.
	\begin{lemma}\label{matchinglemma}
		Let $n=2p$ or $n=2p+1$. If $1 \leq k \leq p-1$, then $M_k$ has a matching that covers all the vertices of $T_{k}$. If $p \leq k \leq n-2$ then $M_k$ has a matching that covers all the vertices of $T_{k+1}$.
	\end{lemma}
	\begin{proof}
		First suppose  $1 \leq k \leq p-1$. Then by Lemma \ref{same union not adjacent}, $M_k$ is a bipartite graph with vertex set $T_k$ and $T_{k+1}$. By Remark \ref{binomial inequality}, we have 
		\begin{center}$| T_k |$ = $n \choose k$  $\leq$ $ n \choose k+1$ = $| T_{k+1}|$. \end{center}
		By Lemma \ref{degree}, observe that $M_k$ is a biregular graph in which all vertices in $T_k$ (respectively, in $T_{k+1}$) have the same vertex degree. Therefore, for any $J \in T_k$ and $J' \in T_{k+1}$, we have 
		\begin{center}
			$n-k \choose 1$ $n \choose k$ = $deg_{M_k}(J)  \cdot | T_k |$ = $deg_{M_k}(J')  \cdot | T_{k+1} |$ = $k+1 \choose k$ $n \choose k+1$. 
		\end{center}
		where $deg_{M_k}(J)$ and $deg_{M_k}(J')$ is the degree of $J$ and $J'$ in the induced subgraph $M_k$ of $\mathcal{I}n(S)$. Thus, $deg_{M_k}(J) \geq deg_{M_k}(J')$. Let $T$ be any arbitrary subset of $T_k$ and consider the induced subgraph of $M_k$ with vertex set $T$ and $N_{M_k}(T)$. The number of edges of this graph is $| T | \cdot deg_{M_k}(L) \leq  | N_{M_k}(T) | \cdot deg_{M_k}(L')$, where $L$ and $L'$ are vertices of $T$ and $N_{M_k}(T)$, respectively. Thus, we have $| T | \leq | N_{M_k}(T) |$. Then by Lemma \ref{hall theorem}, $M_k$ has a matching that covers all the vertices of $T_k$.
		
		The proof for $p \leq k \leq n-2$, is similar. Hence, omitted. 
	\end{proof}
	\begin{theorem}
		Let $S$ be a completely simple semigroup with $n$ minimal left ideals, where $n =2p$ or $n=2p+1$. Then $\alpha(\mathcal{I}n(S))$ =$ n \choose p$.
	\end{theorem}
	\begin{proof}
		Let $n =2p$ or $n= 2p+1$. By Lemma \ref{same union not adjacent}, note that $T_p$ forms an independent set of $\mathcal{I}n(S)$. Consequently, $\alpha(\mathcal{I}n(S))$ $\geq$ $| T_p |$ = $n \choose p$.	Let $\mathcal{U}$ be an arbitrary independent set of $\mathcal{I}n(S)$. We need to show that $| \mathcal{U} | \leq | T_p |$.
		
		If $1 \leq k \leq p-1$, by  Lemma \ref{matchinglemma}, consider $Q_k$ to be a fixed matching of $M_k$ that covers all the vertices of $T_k$. Assume that  $\phi_{k}$ is a mapping from $T_k$ to $T_{k+1}$ which sends a vertex $J \in T_k$ to a vertex $J'$ of $T_{k+1}$ such that $(J, J')$ is an edge in $Q_k$. Since $Q_k$ is a matching of $M_k$, we get $\phi_{k}$ is a one-one map for any $k$. Now, consider $\mathcal{U}_1$ to be $\mathcal{U}$ and  recursively define $\mathcal{U}_2$, $\mathcal{U}_3$, $\ldots$, $\mathcal{U}_k$, $\ldots$, $\mathcal{U}_p$  for $2 \leq k \leq p$ as follows:
		\begin{center}
			$\mathcal{U}_2 = (\mathcal{U}_1 \setminus({\mathcal{U}_1} \cap T_1)) \cup \phi_{1}(\mathcal{U}_1 \cap T_1)$\\
			$\mathcal{U}_3 = (\mathcal{U}_2 \setminus({\mathcal{U}_2} \cap T_2)) \cup \phi_{2}(\mathcal{U}_2 \cap T_2)$\\
			$\vdots$\\
			$\mathcal{U}_k = (\mathcal{U}_{k-1} \setminus({\mathcal{U}_{k-1}}\cap T_{k-1})) \cup \phi_{k-1}(\mathcal{U}_{k-1} \cap T_{k-1})$\\
			$\vdots$\\
			$\mathcal{U}_p = (\mathcal{U}_{p-1} \setminus({\mathcal{U}_{p-1}} \cap T_{p-1})) \cup \phi_{p-1}(\mathcal{U}_{p-1} \cap T_{p-1})$.
		\end{center}
		Observe that $\mathcal{U}_k \cap (T_1 \cup T_2 \cup \cdots \cup T_{k-1}) = \emptyset$, for any $2 \leq k \leq p$.
		First we show that $\mathcal{U}_k$ is an independent set and $| \mathcal{U}_k| = | \mathcal{U}|$. To do this we proceed by induction on $k$, where $k = 1,2, \ldots, p$. Clearly, for $k = 1$, $\mathcal{U}_1$ is an independent set and $|\mathcal{U}_1| = | \mathcal{U}|$ and assume for $\mathcal{U}_{k-1}$ i.e., $\mathcal{U}_{k-1}$ is an independent set and $|\mathcal{U}_{k-1}|= | \mathcal{U}|$. Now, we prove it for $\mathcal{U}_k$. First we show that $|\mathcal{U}_k| = | \mathcal{U}|$. For this purpose we prove that $\phi_{k-1}(\mathcal{U}_{k-1} \cap T_{k-1})$ and $(\mathcal{U}_{k-1} \setminus({\mathcal{U}_{k-1}}\cap T_{k-1}))$ have no common vertices. Assume that there exists $J \in (\mathcal{U}_{k-1} \setminus({\mathcal{U}_{k-1}} \cap T_{k-1})) \cap \phi_{k-1}(\mathcal{U}_{k-1} \cap T_{k-1})$. Then there is $J' \in (\mathcal{U}_{k-1} \cap T_{k-1})$ such that $(J', J) \in Q_{k-1}$ and $J' \subset J$. Since $J \sim J'$ in $\mathcal{U}_{k-1}$, a contradiction. Hence, 
		\begin{center}
			$(\mathcal{U}_{k-1} \setminus({\mathcal{U}_{k-1}} \cap T_{k-1})) \cap \phi_{k-1}(\mathcal{U}_{k-1} \cap T_{k-1}) = \emptyset$.
		\end{center}
		Thus,  $|\mathcal{U}_k| =| \mathcal{U}_{k-1}| =  | \mathcal{U}|$. Now if  $\mathcal{U}_k$ is not an independent set then for any vertex $J \in \phi_{k-1}(\mathcal{U}_{k-1} \cap T_{k-1})$ and $J' \in (\mathcal{U}_{k-1} \setminus({\mathcal{U}_{k-1}} \cap T_{k-1}))$, $J \sim J'$ and $J \subset J'$. Since $Q_k$ is matching, then there exists $J'' \in \mathcal{U}_{k-1} \cap T_{k-1}$ such that $\phi_{k-1}(J'') = J$ and $J'' \subset J$. Consequently, $J'' \sim J'$ in $\mathcal{U}_{k-1}$, a contradiction. Thus, $\mathcal{U}_k$ is an independent set.
		
		For $p \leq k \leq n-2$. By  Lemma \ref{matchinglemma}, consider $Q'_k$ to be a fixed matching of $M_k$ that covers all the vertices of $T_{k+1}$. Assume that  $\phi'_{k}$ is a mapping from $T_{k+1}$ to $T_{k}$ which sends a vertex $J \in T_{k+1}$ to a vertex $J'$ of $T_{k}$ such that $(J, J')$ is an edge in $Q'_k$. As $Q'_k$ is a matching of $M_k$, so $\phi'_{k}$ is one-one map for any $k$. Now, consider $\mathcal{V}_{n-1}$ to be $\mathcal{U}_p$ and analogously define $\mathcal{V}_{n-2}$, $\mathcal{V}_{n-3}$, $\ldots$,    $\mathcal{V}_{p}$ for $p \leq k \leq n-2$ as follows:
		\begin{center}
			$\mathcal{V}_{n-2} = (\mathcal{V}_{n-1} \setminus({\mathcal{V}_{n-1}} \cap T_{n-1})) \cup \phi'_{n-1}(\mathcal{V}_{n-1} \cap T_{n-1})$
			
			$\vdots$
			
			$\mathcal{V}_{k} = (\mathcal{V}_{k-1} \setminus({\mathcal{V}_{k-1}} \cap T_{k-1})) \cup \phi'_{k-1}(\mathcal{V}_{k-1} \cap T_{k-1})$
			
			$\vdots$
			
			$\mathcal{V}_{p} = (\mathcal{V}_{p-1} \setminus({\mathcal{V}_{p-1}} \cap T_{p-1}))\cup \phi'_{p-1}(\mathcal{V}_{p-1} \cap T_{p-1})$.
		\end{center}
		Note that,  $\mathcal{V}_k \cap (T_{n-1} \cup T_{n-2} \cup \cdots \cup T_{k+1}) = \emptyset$, for any $p \leq k \leq n-2$. Similarly, as shown above  we can prove that $\mathcal{V}_k$ is an independent set and $|\mathcal{V}_k| = | \mathcal{V}_{n-1} | = | \mathcal{U}_p |=  | \mathcal{U}|$ for $k= n-2, n-3, \ldots, p$.
		Since $\mathcal{V}_p \subseteq T_p$, we have $| \mathcal{V}_p | \leq | T_p |$. Consequently, we  have $| \mathcal{U} | \leq | T_p |$ = $n \choose p$. Hence, $\alpha(\mathcal{I}n(S))$ =$ n \choose p$.
	\end{proof}

	\begin{corollary}
		Let $S$ be a completely simple semigroup with $n$ minimal left ideals, where $n =2p$ or $n = 2p+1$. Then the vertex covering number is $(2^n-2)-$ $ n \choose p$.
	\end{corollary}
	
	\begin{theorem}\label{perfectmatchinginclusion}
		Let $S$ be a completely simple semigroup with $n$ minimal left ideals. Then $\mathcal{I}n(S)$ has a perfect matching.
	\end{theorem}
	\begin{proof}
		In view of Theorem \ref{order}, it is sufficient to provide a matching of size $2^{n-1}-1$. We have the following cases.
		
		\noindent\textbf{Case 1.} $n = 2p+1$. By Lemma \ref{same union not adjacent}, note that the set of edges \begin{center}
			$M = \{(I_{{i_1}{i_2}\cdots{i_k}}, I_{{j_1}{j_2}\cdots{j_{n-k}})}  : i_1, i_2, \cdots, i_k, j_1, j_2, \cdots, j_{n-k}\in [n] \}$
		\end{center}
		forms a matching and $2|M| = |V(\mathcal{I}n(S))|$. Thus, $|M| = \frac{|V(\mathcal{I}n(S))|}{2}$. Consequently, $\mathcal{I}n(S)$ has a perfect matching.\\
		\noindent\textbf{Case 2.} $n = 2p$. Consider 
		\begin{center}
			$T_1 = \{I_{i_1} : i_1 \in [n] \}$
			
			$T_2 = \{I_{{i_1}{i_2}} : i_1, i_2 \in [n] \}$
			
			$\vdots$
			
			$T_k = \{I_{{i_1}{i_2}\cdots{i_k}} : i_1, i_2, \cdots, i_k \in [n] \}$
			
			$\vdots$
			
			$T_{n-1} = \{I_{{i_1}{i_2}\cdots{i_{n-1}}} : i_1, i_2, \cdots, i_{n-1} \in [n] \}$
		\end{center}
		Note that $T_1, T_2, \ldots, T_{n-1}$  forms a partition of $V(\mathcal{I}n(S))$. Consider the following injective maps
		\begin{center}
			$\phi_{1}: T_1 \setminus \{I_1\} \rightarrow T_2$\\
			$\phi_{2}: T_2 \setminus im(\phi_{1}) \rightarrow T_3$\\
			$\vdots$\\
			$\phi_{n-2}: T_{n-2} \setminus im(\phi_{n-3}) \rightarrow T_{n-1} \setminus \{I_{12 \cdots {(n-1)}}\}$.
		\end{center}
		under the assignment  $J \mapsto J'$ such that $(J, J')$ is an edge in $\mathcal{I}n(S)$.
		The set $M = \{(I_1, I_{12\cdots {(n-1)}})\} \cup \{(\alpha_1, \phi_{1}(\alpha_1)) : \alpha_1 \in dom(\phi_{1})\} \cup \{(\alpha_2, \phi_{2}(\alpha_2)) : \alpha_2 \in dom(\phi_{2})\} \cup \{(\alpha_3, \phi_{3}(\alpha_3)) : \alpha_3 \in dom(\phi_{3})\} \cup \cdots \cup \{(\alpha_{n-2}, \phi_{n-2}(\alpha_{n-2})) : \alpha_{n-2} \in dom(\phi_{n-2})\}$ forms a matching and no edge in $M$ share same end vertices. In the above, by $im(\phi_{i})$ and $dom(\phi_{i})$, we mean the image set and domain set of $\phi_{i}$ respectively. Further, note that $2|M| = |V(\mathcal{I}n(S))|$. Thus, $|M| = \frac{|V(\mathcal{I}n(S))|}{2}$. Consequently, $\mathcal{I}n(S)$ has a perfect matching.
	\end{proof}

	\begin{corollary}
		Let $S$ be a completely simple semigroup with $n$ minimal left ideals. Then the edge covering number is $2^{n-1}-1$.
	\end{corollary}	
	
	\section{Automorphism group}
An automorphism of a graph $\Gamma$  is a permutation $f$ on $V (\Gamma)$ with the property that, for any vertices $u$ and $v$, we
	have $uf \sim vf$ if and only if $u \sim v$. The set $Aut(\Gamma)$ of all graph automorphisms of a graph $\Gamma$ forms a group with
	respect to composition of mappings. A graph $\Gamma$ is \emph{vertex transitive} if for every two vertices $u$ and $v$ there exists a graph automorphism $f$ such that $uf = v$. The symmetric group of degree $n$ is denoted by $S_n$. In order to study algebraic properties of $\mathcal{I}n(S)$, where $S$ is completely simple semigroup, we obtained the automorphism group of $\mathcal{I}n(S)$. For a completely simple semigroup $S$ having two minimal left ideals, 
	$\mathcal{I}n(S)$ is disconnected (cf. Theorem \ref{two minimal}). It follows that    $Aut(\mathcal{I}n(S)) \cong \mathbb{Z}_2$. Now in the remaining section, we find the automorphism group of the inclusion ideal graph of completely simple semigroup having at least three minimal left ideals. In view of Lemma \ref{degree}, we have the following remark.
	
	\begin{remark}\label{degree k}
	In $\mathcal{I}n(S)$, $deg(I_{{i_1}{i_2}\cdots {i_k}}) = deg(I_{{j_1}{j_2}\cdots {j_{n-k}}}) = deg(I_{{j_1}{j_2}\cdots {j_{k}}}).$	
	\end{remark}
	
\begin{lemma}\label{symmetric group}
		For $\sigma \in S_n$, let $\phi_{\sigma} : V(\mathcal{I}n(S)) \rightarrow V(\mathcal{I}n(S))$ defined by $\phi_{\sigma}(I_{{i_1}{i_2}\cdots {i_k}}) = I_{\sigma({i_1})\sigma({i_2})\cdots \sigma({i_k})}$. Then $\phi_{\sigma} \in Aut(\mathcal{I}n(S))$.
	\end{lemma}
	\begin{proof}
Let $I_{{i_1}{i_2}\cdots {i_t}}$ and $I_{{j_1}{j_2}\cdots {j_k}}$ be arbitrary vertices of $\mathcal{I}n(S)$ such that $I_{{i_1}{i_2}\cdots {i_t}} \sim I_{{j_1}{j_2}\cdots {j_k}}$. Without loss of generality, assume that $I_{{i_1}{i_2}\cdots {i_t}} \subset I_{{j_1}{j_2}\cdots {j_k}}$. This implies that $I_{i_1}, I_{i_2}, \ldots, I_{i_t} \subset I_{{j_1}{j_2}\cdots {j_k}}$.
Now 
\begin{align*}
I_{{i_1}{i_2}\cdots {i_t}} \sim I_{{j_1}{j_2}\cdots {j_k}} 
&\Longleftrightarrow I_{\sigma({i_1})\sigma({i_2})\cdots \sigma({i_t})} \sim I_{\sigma({j_1})\sigma({j_2})\cdots \sigma({j_k})}\\
& \Longleftrightarrow \phi_{\sigma}(I_{{i_1}{i_2}\cdots {i_t}}) \sim \phi_{\sigma}(I_{{j_1}{j_2}\cdots {j_k}}).
\end{align*}
Thus, $\phi_{\sigma} \in Aut(\mathcal{I}n(S))$.
	\end{proof}
	
	\begin{lemma}\label{automorphism 1 to n}
		Let $f \in Aut(\mathcal{I}n(S))$ such that  $f(I_{i_s}) = I_{{j_1}{j_2}\cdots{j_{n-1}}}$ for some $i_s \in [n]$. Then $f(I_{{i_1}{i_2}\cdots{i_{k}}}) =  I_{{i_1^{'}}{i_2^{'}}\cdots{i_{n-k}^{'}}}$ for all $I_{{i_1}{i_2}\cdots{i_{k}}} \in V(\mathcal{I}n(S))$.
	\end{lemma}
	\begin{proof}
		Suppose $f(I_{i_s})$ = $I_{{j_1}{j_2}\cdots{j_{n-1}}}$. Since $I_{j_1}$ $\sim I_{{j_1}{j_2}\cdots{j_{n-1}}}$, $I_{j_2}$ $\sim  I_{{j_1}{j_2}\cdots{j_{n-1}}}$, $\ldots$, $I_{j_{n-1}}$ $ \sim I_{{j_1}{j_2}\cdots{j_{n-1}}}$. It follows that $f(I_{j_1})$ $ \sim$ $f(I_{{j_1}{j_2}\cdots{j_{n-1}}})$, $f(I_{j_2}) \sim  f(I_{{j_1}{j_2}\cdots{j_{n-1}}})$, $\ldots$, $f(I_{j_{n-1}}) \sim f(I_{{j_1}{j_2}\cdots{j_{n-1}}})$. By Lemma \ref{degree} and Remark \ref{degree k}, we get either $f(I_{j_1}) = I_{j_\ell}$ or $f(I_{j_1}) = I_{{j_1^{'}}{j_2^{'}}\cdots{j_{n-1}^{'}}}
		$. First suppose that $f(I_{j_1}) = I_{j_\ell}$. Then by Remark \ref{degree k} either $f(I_{{j_1}{j_2}\cdots{j_{n-1}}}) = I_{i_p}$ or $f(I_{{j_1}{j_2}\cdots{j_{n-1}}}) =I_{{\ell_1}{\ell_2}\cdots{\ell_{n-1}}}$. For $f(I_{{j_1}{j_2}\cdots{j_{n-1}}}) = I_{i_p}$ since $I_{j_1} \sim I_{{j_1}{j_2}\cdots{j_{n-1}}}$, we get $I_{j_\ell} = f(I_{j_1}) \sim f(I_{{j_1}{j_2}\cdots{j_{n-1}}}) = I_{i_p}$, a contradiction (see Lemma \ref{same union not adjacent}). Consequently, $f(I_{{j_1}{j_2}\cdots{j_{n-1}}}) =I_{{\ell_1}{\ell_2}\cdots{\ell_{n-1}}}$. Now if $I_{i_s} \sim I_{{j_1}{j_2}\cdots{j_{n-1}}}$ then $I_{{j_1}{j_2}\cdots{j_{n-1}}} = f(I_{i_s}) \sim f(I_{{j_1}{j_2}\cdots{j_{n-1}}}) = I_{{\ell_1}{\ell_2}\cdots{\ell_{n-1}}}$, which is not possible. If
		$I_{i_s} \nsim I_{{j_1}{j_2}\cdots{j_{n-1}}}$ and $I_{i_s} \neq {I_{j_\ell}}$ then $I_{j_1} \nsim I_{i_s}$ implies that $f(I_{j_1}) \nsim f(I_{i_s})$. Consequently, $I_{j_\ell} \nsim I_{{j_1}{j_2}\cdots{j_{n-1}}}$, a contradiction since $j_\ell \in \{j_1, j_2, \ldots, j_{n-1}\}$. We may now suppose that  $I_{j_\ell} = I_{i_s}$. Since $I_{j_2} \sim I_{{j_1}{j_2}\cdots{j_{n-1}}}$, it follows that $f(I_{j_2}) \sim f(I_{{j_1}{j_2}\cdots{j_{n-1}}})$. If $f(I_{j_2}) = I_{i_t}$, since $I_{j_2} \nsim I_{i_s}$ this implies that  $f(I_{j_2}) \nsim f(I_{i_s})$, which is not possible since $i_t \in \{j_1, j_2, \ldots, j_{n-1}\}$. Now if $f(I_{j_2}) = I_{{i_1^{'}}{i_2^{'}}\cdots{i_{n-1}^{'}}}$. Since $I_{j_1} \nsim I_{j_2}$ this implies that $f(I_{j_1}) \nsim f(I_{j_2})$. It follows that $I_{i_s} \nsim I_{{i_1^{'}}{i_2^{'}}\cdots{i_{n-1}^{'}}}$. Since $j_1, j_2, \ldots, j_{n-1} \in [n] \setminus \{i_s\}$. This implies that $j_1, j_2, \ldots, j_{n-1} \in \{i_1^{'}, i_2^{'}, \ldots, i_{n-1}^{'}\}$. Thus $f(I_{i_s}) = f(I_{j_2})$, a contradiction. Therefore, for any $j_1 \in [n]$, $f(I_{j_1}) = I_{{j_1^{'}}{j_2^{'}}\cdots{j_{n-1}^{'}}}$ where $j_1^{'}, j_2^{'}, \ldots, j_{n-1}^{'} \in [n]$.
		
		Next we show that $f(I_{{i_1}{i_2}\cdots{i_{k}}}) = I_{{i_1^{'}}{i_2^{'}}\cdots{i_{n-k}^{'}}}$ for $1 < k \leq n-1$. Since $I_{i_1} \sim I_{{i_1}{i_2}\cdots{i_{k}}}$, $I_{i_2} \sim I_{{i_1}{i_2}\cdots{i_{k}}}$,  $\ldots$, $I_{i_k} \sim I_{{i_1}{i_2}\cdots{i_{k}}}$. This implies that $f(I_{i_1}) \sim f(I_{{i_1}{i_2}\cdots{i_{k}}})$, $f(I_{i_2}) \sim f(I_{{i_1}{i_2}\cdots{i_{k}}})$,  $\ldots$, $f(I_{i_k}) \sim f(I_{{i_1}{i_2}\cdots{i_{k}}})$. Either $f(I_{i_1}) \subset f(I_{{i_1}{i_2}\cdots{i_{k}}})$ or   $f(I_{{i_1}{i_2}\cdots{i_{k}}}) \subset f(I_{i_1})$. Since  $f(I_{i_1}) = I_{{j_1}{j_2}\cdots{j_{n-1}}}$ for some $j_1, j_2, \ldots, j_{n-1} \in [n]$. It follows that $f(I_{i_1}) \not \subset f(I_{{i_1}{i_2}\cdots{i_{k}}})$. Consequently, we get $f(I_{{i_1}{i_2}\cdots{i_{k}}}) \subset f(I_{i_1})$. Thus, there exists $I_{\ell_1} \not \subset f(I_{i_1})$ such that $I_{\ell_1} \not \subset f(I_{{i_1}{i_2}\cdots{i_{k}}})$. Similarly, one can get $f(I_{{i_1}{i_2}\cdots{i_{k}}}) \subset f(I_{i_2})$, $f(I_{{i_1}{i_2}\cdots{i_{k}}}) \subset f(I_{i_3})$, $\ldots$, $f(I_{{i_1}{i_2}\cdots{i_{k}}}) \subset f(I_{i_k})$. Consequently, there exist  $I_{\ell_2} \not \subset f(I_{i_2})$,  $I_{\ell_3} \not \subset f(I_{i_3})$, $\ldots$,  $I_{\ell_k} \not \subset f(I_{i_k})$ such that $I_{\ell_2}, I_{\ell_3}, \ldots, I_{\ell_k} \not \subset f(I_{{i_1}{i_2}\cdots{i_{k}}})$. Thus, $I_{{\ell_1}{\ell_2}\cdots{\ell_k}} \not \subset f(I_{{i_1}{i_2}\cdots{i_{k}}})$. Clearly, there exist $i_{k+1}, i_{k+2}, \ldots, i_{n-k} \notin \{i_1, i_2, \ldots, i_k\}$ such that $I_{i_{k+1}}, I_{i_{k+2}}, \ldots, I_{i_{n-k}} \nsim I_{{i_1}{i_2}\cdots{i_{k}}}$. It follows that $f(I_{i_{k+1}}) \nsim f(I_{{i_1}{i_2}\cdots{i_{k}}})$, $f(I_{i_{k+2}}) \nsim f(I_{{i_1}{i_2}\cdots{i_{k}}})$, $\ldots$, $f(I_{i_{n-k}}) \nsim f(I_{{i_1}{i_2}\cdots{i_{k}}})$ so that $f(I_{{i_1}{i_2}\cdots{i_{k}}}) \not \subset f(I_{i_{k+1}})$, $f(I_{{i_1}{i_2}\cdots{i_{k}}}) \not \subset f(I_{i_{k+2}})$, $\ldots$, $f(I_{{i_1}{i_2}\cdots{i_{k}}}) \not \subset f(I_{i_{n-k}})$. Thus, there exists $I_{{i_1^{'}}} \not \subset f(I_{i_{{k+1}}})$ such that $I_{{i_1^{'}}} \subset f(I_{{i_1}{i_2}\cdots{i_{k}}})$. Similarly,  there exist $I_{{i_2^{'}}} \subset f(I_{{i_1}{i_2}\cdots{i_{k}}})$, $I_{{i_3^{'}}} \subset f(I_{{i_1}{i_2}\cdots{i_{k}}})$, $\ldots$, $I_{{i_{n-k}^{'}}} \subset f(I_{{i_1}{i_2}\cdots{i_{k}}})$ such that $I_{{i_2^{'}}}, I_{{i_3^{'}}}, \ldots, I_{{i_{n-k}^{'}}} \subset f(I_{{i_1}{i_2}\cdots{i_{k}}})$. It follows that $I_{{i_1^{'}}{i_2^{'}}\cdots{i_{n-k}^{'}}} \subset f(I_{{i_1}{i_2}\cdots{i_{k}}})$ and $I_{{\ell_1}{\ell_2}\cdots{\ell_k}} \not \subset f(I_{{i_1}{i_2}\cdots{i_{k}}})$. Thus, $f(I_{{i_1}{i_2}\cdots{i_{k}}}) =  I_{{i_1^{'}}{i_2^{'}}\cdots{i_{n-k}^{'}}}$.
		\end{proof}
	\begin{lemma}\label{alpha}
		Let  $\alpha : V(\mathcal{I}n(S)) \rightarrow V(\mathcal{I}n(S))$ be a mapping defined by   $\alpha(I_{{i_1}{i_2}\cdots{i_{k}}}) = I_{{i_1^{'}}{i_2^{'}}\cdots{i_{n-k}^{'}}}$ such that $i_1^{'}, i_2^{'}, \ldots, i_{n-k}^{'}  \in [n] \setminus \{i_1, i_2, \ldots, i_k \}$. Then $\alpha \in Aut(\mathcal{I}n(S))$.
	\end{lemma}
	\begin{proof}
		It is straightforward to verify that $\alpha$ is one-one  and onto map on $V(\mathcal{I}n(S))$. Note that for any $I_{{j_1}{j_2}\cdots{j_t}}, I_{{j_1^{'}}{j_2^{'}}\cdots{j_s^{'}}} \in V(\mathcal{I}n(S))$. Suppose that  $I_{{j_1}{j_2}\cdots{j_t}} \sim I_{{j_1^{'}}{j_2^{'}}\cdots{j_{s}^{'}}}$. Without loss of generality, assume that $I_{{j_1}{j_2}\cdots{j_t}} \subset I_{{j_1^{'}}{j_2^{'}}\cdots{j_{s}^{'}}}$. Thus, $j_1, j_2, \ldots, j_t  \in \{j_1^{'}, j_2^{'}, \ldots, j_{s}^{'}\}$. If $\alpha(I_{{j_1}{j_2}\cdots{j_t}}) = I_{{l_1}{l_2} \cdots {l_{n-t}}}$, where $j_1, j_2, \ldots, j_t \in [n] \setminus \{l_1, l_2, \ldots, l_{n-t} \}$   
		and $\alpha(I_{{j_1^{'}}{j_2^{'}}\cdots{j_{s}^{'}}}) = I_{{l_1^{'}}{l_2^{'}} \cdots {l_{n-s}^{'}}}$, where $j_1^{'}, j_2{'}, \ldots, j_s^{'} \in [n] \setminus \{l_1^{'}, l_2^{'}, \ldots, {l_{n-s}^{'}} \}$. Since  $l_1, l_2, \ldots, l_{n-t} \in [n] \setminus \{j_1, j_2, \ldots, j_t\}$ and $l_1^{'}, l_2^{'}, \ldots, {l_{n-s}^{'}}  \in [n] \setminus \{j_1^{'}, j_2{'}, \ldots, j_s^{'}\}$. This implies that $\{l_1^{'}, l_2^{'}, \ldots, {l_{n-s}^{'}} \} \subset \{l_1, l_2, \ldots, l_{n-t} \}$. It follows that $I_{{l_1^{'}}{l_2^{'}} \cdots {l_{n-s}^{'}}} \subset I_{{l_1}{l_2} \cdots {l_{n-t}}}$. Consequently, $\alpha(I_{{j_1^{'}}{j_2^{'}}\cdots{j_{s}^{'}}}) \subset \alpha(I_{{j_1}{j_2}\cdots{j_t}})$. Thus, $ I_{{j_1}{j_2}\cdots{j_t}} \sim I_{{j_1^{'}}{j_2^{'}}\cdots{j_{s}^{'}}}$ implies that $\alpha(I_{{j_1}{j_2}\cdots{j_t}}) \sim \alpha(I_{{j_1^{'}}{j_2^{'}}\cdots{j_{s}^{'}}}).$ Now, if $\alpha(I_{{j_1}{j_2}\cdots{j_t}}) \sim \alpha(I_{{j_1^{'}}{j_2^{'}}\cdots{j_{s}^{'}}})$. Without loss of generality, assume that  $\alpha(I_{{j_1}{j_2}\cdots{j_t}}) \subset \alpha(I_{{j_1^{'}}{j_2^{'}}\cdots{j_{s}^{'}}})$. Similar to the argument discussed above we obtain if $\alpha(I_{{j_1}{j_2}\cdots{j_t}}) \subset \alpha(I_{{j_1^{'}}{j_2^{'}}\cdots{j_{s}^{'}}})$ then $I_{{j_1^{'}}{j_2^{'}}\cdots{j_{s}^{'}}} \subset I_{{j_1}{j_2}\cdots{j_t}}$.  
		Thus, $\alpha$ is an automorphism.
	\end{proof}
	\begin{remark}
		For $\phi_{\sigma}$ and $\alpha$, defined in Lemma \ref{symmetric group} and \ref{alpha}, we have  $\phi_{\sigma} \circ \alpha = \alpha \circ \phi_{\sigma}$.	
	\end{remark}
	\begin{proposition}\label{a and S}
		For each $f \in  Aut(\mathcal{I}n(S))$, we have either $f = \phi_{\sigma}$ or $f = \phi_{\sigma} \circ \alpha$ for some $\sigma \in S_n$.
	\end{proposition}
	\begin{proof}
	In view of Remark \ref{degree k} and Lemma \ref{automorphism 1 to n}, we prove the result through the following cases.
	
		\noindent\textbf{Case 1.} $f(I_{i_1}) = I_{j_1}$, $f(I_{i_2}) = I_{j_2}$, $\ldots$, $f(I_{i_n}) = I_{j_n}$. Consider $\sigma \in S_n$ such that $\sigma(i_1) = j_1, \sigma(i_2) = j_2, \ldots, \sigma(i_n) = j_n$. Then $\phi_{\sigma}(I_{{i_1}{i_2}\cdots{i_k}}) = I_{\sigma({i_1})\sigma({i_2})\cdots \sigma({i_k})} = I_{{j_1}{j_2}\cdots{j_k}}$ (cf. Lemma \ref{symmetric group}). Clearly, $I_{i_1} \sim I_{{i_1}{i_2}\cdots{i_k}}$, $I_{i_2} \sim I_{{i_1}{i_2}\cdots{i_k}}$, $\ldots$, $I_{i_k} \sim I_{{i_1}{i_2}\cdots{i_k}}$. Also note that for $i_{k+1}, i_{k+2}, \ldots, i_{n} \in [n] \setminus \{i_1, i_2, \ldots, i_k\}$, we have  $I_{i_{k+1}} \nsim I_{{i_1}{i_2}\cdots{i_k}}$, $I_{i_{k+2}} \nsim I_{{i_1}{i_2}\cdots{i_k}}$, $\ldots$, $I_{i_{n}} \nsim I_{{i_1}{i_2}\cdots{i_k}}$. Thus, $f(I_{i_1}) \sim f(I_{{i_1}{i_2}\cdots{i_k}})$, $f(I_{i_2}) \sim f(I_{{i_1}{i_2}\cdots{i_k}})$, $\ldots$, $f(I_{i_k}) \sim f(I_{{i_1}{i_2}\cdots{i_k}})$ and $f(I_{i_{k+1}}) \nsim f(I_{{i_1}{i_2}\cdots{i_k}})$, $f(I_{i_{k+2}}) \nsim f(I_{{i_1}{i_2}\cdots{i_k}})$, $\ldots$, $f(I_{i_{n}}) \nsim f(I_{{i_1}{i_2}\cdots{i_k}})$. Consequently, $I_{j_1} \subset f(I_{{i_1}{i_2}\cdots{i_k}})$, $I_{j_2} \subset f(I_{{i_1}{i_2}\cdots{i_k}})$, $\ldots$, $I_{j_k} \subset f(I_{{i_1}{i_2}\cdots{i_k}})$ and $I_{j_{k+1}} \not \subset f(I_{{i_1}{i_2}\cdots{i_k}})$, $I_{j_{k+2}} \not \subset f(I_{{i_1}{i_2}\cdots{i_k}})$, $\ldots$, $I_{j_n} \not \subset f(I_{{i_1}{i_2}\cdots{i_k}})$. It follows that $f(I_{{i_1}{i_2}\cdots{i_k}}) = I_{{j_1}{j_2}\cdots{j_k}} = \phi_{\sigma}(I_{{i_1}{i_2}\cdots{i_k}})$. Thus, $f = \phi_{\sigma}$.
		
		\noindent\textbf{Case 2.} $f(I_{i_1}) = I_{{j_1}{j_2}\cdots{j_{n-1}}}$, $f(I_{i_2}) = I_{{j_1^{'}}{j_2^{'}}\cdots{j_{n-1}^{'}}}$, $\ldots$, $f(I_{i_n}) = I_{{\ell_1}{\ell_2}\cdots{\ell_{n-1}}}$. Assume that  $I_{i_1^{'}} \not \subset f(I_{i_1})$, $I_{i_2^{'}}\not \subset f(I_{i_2})$, $\ldots$, $I_{i_n^{'}} \not \subset f(I_{i_n})$. Since $I_{i_1} \sim I_{{i_1}{i_2}\cdots{i_k}}$, $I_{i_2} \sim I_{{i_1}{i_2}\cdots{i_k}}$, $\ldots$, $I_{i_k} \sim I_{{i_1}{i_2}\cdots{i_k}}$, we obtain $f(I_{i_1}) \sim f(I_{{i_1}{i_2}\cdots{i_k}})$, $f(I_{i_2}) \sim f(I_{{i_1}{i_2}\cdots{i_k}})$, $\ldots$, $f(I_{i_k}) \sim f(I_{{i_1}{i_2}\cdots{i_k}})$. Consequently, $I_{i_1^{'}}\not \subset f(I_{{i_1}{i_2}\cdots{i_k}})$, $I_{i_2^{'}}\not \subset f(I_{{i_1}{i_2}\cdots{i_k}})$, $\ldots$, $I_{i_k^{'}}\not \subset f(I_{{i_1}{i_2}\cdots{i_k}})$. It follows that $I_{{i_1^{'}}{i_2^{'}}\cdots{i_k^{'}}} \not \subset f(I_{{i_1}{i_2}\cdots{i_k}})$. For $i_{k+1}, i_{k+2},\ldots, i_{n} \in [n] \setminus \{i_1, i_2, \ldots, i_k\}$, we have  $I_{i_{k+1}} \nsim I_{{i_1}{i_2}\cdots{i_k}}$, $I_{i_{k+2}} \nsim I_{{i_1}{i_2}\cdots{i_k}}$, $\ldots$, $I_{i_{n}} \nsim I_{{i_1}{i_2}\cdots{i_k}}$ so that $f(I_{i_{k+1}}) \nsim f(I_{{i_1}{i_2}\cdots{i_k}})$, $f(I_{i_{k+2}}) \nsim f(I_{{i_1}{i_2}\cdots{i_k}})$, $\ldots$, $f(I_{i_{n}}) \nsim f(I_{{i_1}{i_2}\cdots{i_k}})$. This implies that $I_{i_{k+1}^{'}} \subset f(I_{{i_1}{i_2}\cdots{i_k}})$, $I_{i_{k+2}^{'}} \subset f(I_{{i_1}{i_2}\cdots{i_k}})$, $\ldots$, $I_{i_n^{'}} \subset f(I_{{i_1}{i_2}\cdots{i_k}})$. As a result,  $I_{{i_{k+1}^{'}}{i_{k+2}^{'}}\cdots{i_{n}^{'}}}  \subset f(I_{{i_1}{i_2}\cdots{i_k}})$ and $I_{{i_1^{'}}{i_2^{'}}\cdots{i_k^{'}}} \not \subset f(I_{{i_1}{i_2}\cdots{i_k}})$.  Thus, $f(I_{{i_1}{i_2}\cdots{i_k}}) = I_{{i_{k+1}^{'}}{i_{k+2}^{'}}\cdots{i_{n}^{'}}}$. Define $\sigma(i_1) = i_1^{'}$, $\sigma(i_2) = i_2^{'}$, $\ldots$, $\sigma(i_n) = i_n^{'}$.
		Now, $(\phi_{\sigma} \circ \alpha)(I_{{i_1}{i_2}\cdots{i_k}}) = \phi_{\sigma}(I_{{i_{k+1}}{i_{k+2}}\cdots{i_{n}}}) = I_{{\sigma{(i_{k+1})}}{\sigma(i_{k+2})}\cdots{\sigma(i_{n})}} = I_{{i_{k+1}^{'}}{i_{k+2}^{'}}\cdots{i_{n}^{'}}} = f(I_{{i_1}{i_2}\cdots{i_k}})$. Hence $f = \phi_{\sigma} \circ \alpha$.
	\end{proof}
	\begin{theorem}\label{automorphism group}
		Let $S$ be a completely simple semigroup with $n$ minimal left ideals. Then for $n \geq 3$, we have $ Aut(\mathcal{I}n(S)) \cong S_n \times \mathbb{Z}_2$. Moreover, $|Aut(\mathcal{I}n(S))| = 2(n!)$.
	\end{theorem}
	
	\begin{proof}
			In view of Lemmas \ref{same union not adjacent}, \ref{automorphism 1 to n} and by Proposition \ref{a and S}, 
		note that the underlying set of the automorphism group of $\mathcal{I}n(S)$ is
		$Aut(\mathcal{I}n(S)) = \{\phi_{\sigma} \; : \; \sigma \in S_n \} \cup \{\phi_{\sigma} \circ \alpha \; : \; \sigma \in S_n \}$, where $S_n$ is a symmetric group of degree $n$. The groups $Aut(\mathcal{I}n(S))$ and $S_n \times \mathbb Z_2$ are isomorphic under the assignment $\phi_{\sigma} \mapsto (\sigma, \bar{0})$ and $\phi_{\sigma} \circ \alpha \mapsto (\sigma, \bar{1})$.  Since all the elements in $Aut(\mathcal{I}n(S))$ are distinct, we have $|Aut(\mathcal{I}n(S))| = 2(n!)$.
	\end{proof}
		\begin{theorem}\label{vertex transitive2}
		Let $S$ be a completely simple semigroup with $n$ minimal left ideals. Then the graph  $\mathcal{I}n(S)$ is vertex transitive  if and only if $n \in \{2, 3\}$.
	\end{theorem}
	\begin{proof}
Suppose that $n \geq 4$. Then by Remark \ref{degree k}, there exist at least two vertices whose degree is not equal. Thus, $\mathcal{I}n(S)$ is not a regular graph and so  is not a vertex transitive graph. Conversely, suppose that $n \in \{2, 3\}$. If $n =2$, then we have $V(\mathcal{I}n(S)) = \{I_1, I_2\}$. Then by Lemma \ref{symmetric group}, $\mathcal{I}n(S)$ is vertex transitive. If $n =3$, then we have $V(\mathcal{I}n(S)) = \{I_1, I_2, I_3, I_{12}, I_{13}, I_{23}\}$. Let $J$ and $J'$ be two nontrivial left ideals of $S$. If both $J$ and $J'$ are minimal (or non minimal), then by Theorem \ref{automorphism group}, there exist a graph automorphism $\phi_{\sigma}$ such that $\phi_{\sigma}(J) = J'$. Now suppose that one of them is minimal. Without loss of generality, assume that $J$ is minimal and $J'$ is not a minimal left ideal of $S$. Then again  by Theorem \ref{automorphism group}, there exist a graph automorphism  $\phi_{\sigma} \circ \alpha$ for some $\sigma \in S_n$ such that $(\phi_{\sigma} \circ \alpha)(J) = J'$. Thus, $\mathcal{I}n(S)$ is vertex transitive.
	\end{proof}
	
	Since every connected vertex transitive graph is a retract of Cayley graph (cf. \cite{godsil2013algebraic}), by Theorem \ref{two minimal} and \ref{vertex transitive2}, we have the following corollary.

\begin{corollary}
Let $S$ be a completely simple semigroup with $3$ minimal left ideals. Then the graph  $\mathcal{I}n(S)$ is a retract of Cayley graph.
\end{corollary}

\section{Acknowledgement}
We are thankful to the referees for  valuable suggestions which helped in improving the presentation of the paper.  The first author gratefully acknowledge for providing financial support to CSIR  (09/719(0093)/2019-EMR-I) government of India. The second author wishes to acknowledge the support of MATRICS Grant  (MTR/2018/000779) funded by SERB, India.
		

\end{document}